\documentclass[11pt,a4paper,oneside]{article}
 \pdfoutput=1                        
\usepackage{a4wide}
\usepackage{amsmath,amsfonts,amssymb,amsthm,amscd}
\usepackage{epsfig}
\usepackage{color}
\usepackage{hyperref} 
\usepackage{ifpdf}
\usepackage{microtype} 
\usepackage{array,booktabs}
\usepackage{enumitem}

\if 10
\usepackage[inline]{showlabels}
\definecolor{zelena}{rgb}{0,.35,0}

\showlabels{cite}
\showlabels{bibitem}

\else

\fi

\newtheorem{theorem}{Theorem}       
\newtheorem{proposition}[theorem]{Proposition}           
\newtheorem{lemma}[theorem]{Lemma}
\newtheorem{corollary}[theorem]{Corollary}

\newtheorem{observation}[theorem]{Observation}
\newtheorem{claim}[theorem]{Claim}
\newtheorem{problem}{Problem}

\DeclareMathOperator{\R}{R}
\DeclareMathOperator{\Ro}{\overline{R}}
\DeclareMathOperator{\Rc}{Rc}
\DeclareMathOperator{\Rnoncross}{\overline{Rnc}}
\DeclareMathOperator{\Rg}{Rg}
\DeclareMathOperator{\Es}{ES}
\DeclareMathOperator{\ex}{ex}

\def\inst#1{$^{#1}$}

\begin{document}

\title{Ramsey numbers of ordered graphs\footnote{An extended abstract of this paper appeared in the proceedings of Eurocomb 2015~\cite{bckk15}.} \thanks{The first and the fourth author were supported by the grants SVV-2013-267313 (Discrete Models and Algorithms), GAUK 1262213 of the Grant Agency of Charles University and by the project CE-ITI (GA\v CR P202/12/G061) of the Czech Science Foundation. The third author acknowledges the support from the European Research Council under the European Union's Seventh Framework Programme (FP/2007-2013) / ERC Grant Agreement n. 616787. The fourth author was also supported by ERC Advanced Research Grant no 267165 (DISCONV) and by Swiss National Science Foundation Grants 200021-137574 and 200020-14453.
}
}

\author{
Martin Balko\inst{1} \and Josef Cibulka\inst{2} \and Karel Kr\'al\inst{2} \and Jan Kyn\v{c}l\inst{1, 3, 4}
} 

\date{}

\maketitle

\begin{center}
{\footnotesize
\inst{1} 
Department of Applied Mathematics and Institute for Theoretical Computer Science, \\
Charles University, Faculty of Mathematics and Physics, \\
Malostransk\'e n\'am.~25, 118~00~ Praha 1, Czech Republic; \\
\texttt{balko@kam.mff.cuni.cz, cibulka@kam.mff.cuni.cz, kyncl@kam.mff.cuni.cz}
\\\ \\
\inst{2} 
Department of Applied Mathematics, \\
Charles University, Faculty of Mathematics and Physics, \\
Malostransk\'e n\'am.~25, 118~00~ Praha 1, Czech Republic; \\
\texttt{kralka@kam.mff.cuni.cz}
\\\ \\
\inst{3}
Alfr\'ed R\'enyi Institute of Mathematics, Re\'altanoda u. 13-15, Budapest 1053, Hungary
\\ \ \\
\inst{4}
\'Ecole Polytechnique F\'ed\'erale de Lausanne, Chair of Combinatorial Geometry, \\
EPFL-SB-MATHGEOM-DCG, Station 8, CH-1015 Lausanne, Switzerland\\
}
\end{center}

\begin{abstract}
An \emph{ordered graph} is a pair $\mathcal{G}=(G,\prec)$ where $G$ is a graph and $\prec$ is a total ordering of its vertices.
The \emph{ordered Ramsey number} $\Ro(\mathcal{G})$ is the minimum number $N$ such that every ordered complete graph with $N$ vertices and with edges colored by two colors contains a monochromatic copy of $\mathcal{G}$. 

In contrast with the case of unordered graphs, we show that there are arbitrarily large ordered matchings $\mathcal{M}_n$ on $n$ vertices for which $\Ro(\mathcal{M}_n)$ is superpolynomial in $n$. This implies that ordered Ramsey numbers of the same graph can  grow superpolynomially in the size of the graph in one ordering and remain linear in another ordering.

We also prove that the ordered Ramsey number $\Ro(\mathcal{G})$ is polynomial in the number of vertices of $\mathcal{G}$ if the bandwidth of $\mathcal{G}$ is constant or if $\mathcal{G}$ is an ordered graph of constant degeneracy and constant interval chromatic number.
The first result gives a positive answer to a question of Conlon, Fox, Lee, and Sudakov.

For a few special classes of ordered paths, stars or matchings, we give asymptotically tight bounds on their ordered Ramsey numbers. For so-called monotone cycles we compute their ordered Ramsey numbers exactly. This result implies exact formulas for geometric Ramsey numbers of cycles introduced by K\'arolyi, Pach, T\'oth, and Valtr. 
\end{abstract}

\section{Introduction}
\label{section_intro}

Ramsey's theorem~\cite{rams30} states that for every given graph $G$, every sufficiently large complete graph with edges colored by a constant number of colors contains a monochromatic copy of~$G$. 
We study the analogue of Ramsey's theorem for graphs with ordered vertex sets. 
The concept of ordered graphs appeared earlier in the literature~\cite{klazar04a, klazar04b, milans12, pach06}, but we are not aware of any Ramsey-type results for such graphs except for the case of monotone paths and hyperpaths~\cite{choudum02,elias11,fox12,milans12, moshkovitz12}.

The main goal of this paper is to understand the effects of different vertex orderings on the ordered Ramsey number of a given graph, and to compare the ordered and unordered Ramsey numbers.
We state our results after introducing the necessary notation.
Then we present a few examples that provide the motivation to study ordered Ramsey numbers.

During the preparation of this paper, we learned that Conlon, Fox, Lee, and Sudakov~\cite{conFox14} have independently obtained results that overlap with ours. Ramsey numbers of ordered graphs are
also discussed by Conlon, Fox, and Sudakov in their survey of recent
developments in Ramsey theory~\cite{conFox15}.

Throughout the paper, we omit the ceiling and floor signs whenever they are not crucial.
Unless indicated otherwise, all logarithms in this paper are base 2.

\paragraph{Graphs.}
We consider only finite graphs with no multiple edges and no loops. 
A \emph{coloring} of a graph $G=(V,E)$ is a mapping $f \colon E \to C$ where $C$ is a finite set of \emph{colors}.
Unless specified otherwise, we assume that $C=[c]=\{1,2,\dots,c\}$. 
A coloring with $c$ colors is called a \emph{$c$-coloring}. 
In a 2-coloring of a graph $G$ with colors red and blue, we call a vertex $u$ of $G$ a \emph{red neighbor} (a \emph{blue neighbor}) of a vertex $v$ of $G$ if the edge $uv$ is colored red (blue, respectively). 

Ramsey's theorem states that for given positive integers $c$ and $n$, there is an integer $N$ such that every $c$-coloring of $K_N$ contains a monochromatic copy of~$K_n$.
The minimum such $N$ is called the \emph{Ramsey number} and we denote it by $\R(K_n;c)$. 
Classical results of Erd\H{o}s~\cite{erdos47} and Erd\H{o}s and Szekeres~\cite{ErSz35_a_comb_problem} give the exponential bounds 
\begin{equation}
\label{eq_klasicke_odhady_ramsey} 
2^{n/2} \le \R(K_n;2) \le 2^{2n}.
\end{equation}
Despite many improvements during the last sixty years (see~\cite{conlon09} for example), the constant factors in the exponents have remained the same.

Since every graph on $n$ vertices is contained in $K_n$, we can consider the following generalization of Ramsey numbers.
Let $c$ be a positive integer and let ${G}_1,\dots,{G}_c$ be graphs.
Ramsey's theorem then implies that there exists a smallest number $\R({G}_1,\dots,{G}_c)$ such that every $c$-coloring of a complete graph with at least $\R({G}_1,\dots,{G}_c)$ vertices contains a monochromatic copy of ${G}_i$ in color $i$ for some $i \in [c]$.
The case when all the graphs ${G}_1,\dots,{G}_c$ are isomorphic to ${G}$ is 
called the \emph{diagonal case} and we just write $\R({G};c)$ instead of $\R({G}_1,\dots,{G}_c)$.
We also abbreviate $\R(G;2)$ as $\R(G)$.

We note that the definitions above can be generalized to hypergraphs of uniformity $k$ larger than $2$.
In particular, for a $k$-uniform hypergraph $H$, it is known that its Ramsey number $R_k(H)$ is finite.

\paragraph{Ordered graphs.}
An \emph{ordered graph} $\mathcal{G}$ is a pair $({G},\prec)$ where $G$ is a graph and $\prec$ is a total ordering of its vertex set.
The ordering $\prec$ is called a \emph{vertex ordering} of $G$.
Many notions related to graphs, such as vertex degrees or  a coloring, can be defined analogously for ordered graphs.

For an ordered graph $\mathcal{G}=(G,\prec)$ and its vertices $x,y$, we say that $x$ is a \emph{left neighbor} of $y$ and that $y$ is a \emph{right neighbor} of $x$ if $x$ and $y$ belong to a common edge and $x\prec y$.
We say that two ordered graphs $(G_1,\prec_1)$ and $(G_2,\prec_2)$ are \emph{isomorphic} if $G_1$ and $G_2$ are isomorphic via a one-to-one correspondence $g\colon V(G_1) \to V(G_2)$ that also preserves the orderings; that is, for every $x,y \in V(G_1)$, $x\prec_1 y \Leftrightarrow g(x)\prec_2 g(y)$.
An ordered graph $\mathcal{H}=({H},\prec_1)$ is an \emph{ordered subgraph} of $\mathcal{G}=({G},\prec_2)$, written $\mathcal{H}\subseteq\mathcal{G}$, if ${H}$ is a subgraph of ${G}$ and $\prec_1$ is a suborder of $\prec_2$.

We now introduce Ramsey numbers of ordered graphs, called \emph{ordered Ramsey numbers}.
For given ordered graphs $\mathcal{G}_1,\dots,\mathcal{G}_c$, we denote by $\Ro(\mathcal{G}_1,\dots,\mathcal{G}_c)$ the smallest number $N$ such that every $c$-coloring of $\mathcal{K}_N$ contains, for some $i$, a monochromatic copy of $\mathcal{G}_i$ in color $i$ as an ordered subgraph. 
If all $\mathcal{G}_i$ are isomorphic to $\mathcal{G}$, we write the ordered Ramsey number as $\Ro(\mathcal{G};c)$.
We abbreviate $\Ro(\mathcal{G};2)$ as $\Ro(\mathcal{G})$.
If a coloring $f$ of an ordered graph $\mathcal{G}$ contains no monochromatic copy of $\mathcal{H}$, we say that $f$ \emph{avoids} $\mathcal{H}$.

Up to isomorphism, there is only one ordered complete graph on $n$ vertices, which we denote by $\mathcal{K}_n$. Hence, for arbitrary positive integers $c,r_1,\dots,r_c$ we have $\Ro(\mathcal{K}_{r_1},\dots,\mathcal{K}_{r_c})=\R(K_{r_1},\dots,K_{r_c})$.
Since every ordered graph on $r$ vertices is an ordered subgraph of $\mathcal{K}_r$, we have $\Ro(\mathcal{G}_1,\dots,\mathcal{G}_c) \le \Ro(\mathcal{K}_{r_1},\dots,\mathcal{K}_{r_c})$ where $r_i$ is the number of vertices of $\mathcal{G}_i$.
We have thus proved the following fact.

\begin{observation}
\label{obs_uplnaky}
Let $c$ be an arbitrary positive integer and let $\mathcal{G}_1=(G_1,\prec_1),\allowbreak\dots,\mathcal{G}_c=(G_c,\prec_c)$ be an arbitrary collection of ordered graphs.
Then we have \[\R(G_1,\dots,G_c) \le \Ro(\mathcal{G}_1,\dots,\mathcal{G}_c) \le  \R\left(K_{|V(G_1)|},\dots,K_{|V(G_c)|}\right).\] 
\qedhere
\end{observation}

To study the asymptotic growth of ordered Ramsey numbers, we introduce \emph{ordering schemes} for some classes of graphs.
An ordering scheme is a particular rule for ordering the vertices of the graphs consistently in the whole class. 
For example, a \emph{monotone path} $(P_n,\lhd_{mon})$ is an ordered graph with vertices $v_1 \lhd_{mon} \dots \lhd_{mon} v_n$ and $n-1$ edges, each consisting of two consecutive vertices.
Throughout the paper we use the symbol $\lhd$ instead of $\prec$ to emphasize the fact that the vertex ordering follows some ordering scheme.

For an ordered graph $(G,\prec)$, we say that a vertex $v$ of $G$ is to the \emph{left} (\emph{right}) of a subset $U$  of vertices of $G$ if $v$ precedes (is preceded by, respectively) every vertex of $U$ in $\prec$.
More generally, for two subsets $U$ and $W$ of vertices of $G$, we say that $U$ is \emph{to the left of $W$} and $W$ is \emph{to the right of $U$} if every vertex of $U$ precedes every vertex of $W$ in $\prec$.
We say that a subset $I$ of vertices of $G$ is an \emph{interval} if for every pair of vertices $u$ and $v$ of $I$ such that $u\prec v$, every vertex $w$ of $G$ satisfying $u\prec w\prec v$ is contained in $I$.

Again, all the definitions, as well as Observation~\ref{obs_uplnaky}, can be extended to $k$-uniform hypergraphs with $k>2$.
In particular, we know that the ordered Ramsey number $\Ro_k(\mathcal{H})$ of an arbitrary ordered $k$-uniform hypergraph $\mathcal{H}$ is finite.

\subsection{Our results}

We are interested in the effects of vertex orderings on the ordered Ramsey numbers of various classes of graphs.
The Ramsey number of a graph and the ordered Ramsey number of its ordering can be asymptotically different: for example, Proposition~\ref{prop_mon_paths} implies that $\Ro((P_n,\lhd_{mon}))$ is quadratic while $\R(P_n)$ is linear~\cite{gerencser67}.

The gap is much wider for hypergraphs.
Let $t_h$ denote the \emph{tower function of height} $h$ defined by $t_1(x)=x$ and $t_h(x)=2^{t_{h-1}(x)}$.
It is known that for positive integers $\Delta$ and $k$, there exists a constant $C(\Delta,k)$ such that if $H$ is a $k$-uniform hypergraph with $n$ vertices and maximum degree $\Delta$, then $\R_k(H) \le C(\Delta,k)\cdot n$~\cite{conlonFox09}.
On the other hand, Moshkovitz and Shapira~\cite{moshkovitz12} showed that for every $k \ge 3$ we have $\Ro_k((P^k_n,\lhd_{mon}))=t_{k-1}(2n-o(n))$, where $(P^k_n,\lhd_{mon})$ is a \emph{$k$-uniform tight monotone hyperpath} on $n$ vertices, a natural generalization of monotone paths to $k$-uniform hypergraphs.

In Section~\ref{section_specific}, we derive lower and upper bounds on ordered Ramsey numbers of a few basic classes of ordered graphs: stars, paths and cycles. First, in Subsection~\ref{subsection_stars} we show that ordered Ramsey numbers of all ordered stars are linear with respect to the number of vertices and at most exponential with respect to the number of colors. 

\begin{theorem}
\label{thmStars}
For all integers $c \ge 2$, $n_1,\dots,n_c \ge 3$, and for every collection of ordered stars $\mathcal{S}_1,\dots,\mathcal{S}_c$ where $n_i$ is the number of vertices of $\mathcal{S}_i$, we have 
\[
\Ro(\mathcal{S}_1,\dots,\mathcal{S}_c) \le 2^c+ 2^{c+1}\cdot\sum_{i=1}^{c}(n_i-3) < 2^{c+1}\cdot\sum_{i=1}^{c}n_i.
\]
\end{theorem}

For certain ordered stars we also provide an almost matching lower bound.
In Subsection~\ref{subsection_paths}, we show an ordering $(P_n,\lhd_{alt})$ of the path $P_n$ whose ordered Ramsey number is linear in $n$.

\begin{proposition}
\label{prop_alt_paths}
For every integer $n>2$, we have \[5\lfloor n/2 \rfloor -4 \le \Ro((P_n,\lhd_{alt})) \le 2n-3+\sqrt{2n^2-8n+11}.\]
\end{proposition}

In Subsection~\ref{subsection_cycles} we study ordered cycles.
First, we compute Ramsey numbers for all possible orderings of $C_4$ (Proposition~\ref{propC4}).
Then we derive an exact formula for ordered Ramsey numbers of monotone cycles. 
A \emph{monotone cycle} $(C_n,\lhd_{mon})$ on $n$ vertices consists of a monotone path $v_1 \lhd_{mon} \dots \lhd_{mon} v_n$ and the edge $v_1v_n$; see part~a) of Figure~\ref{fig6_mon_cycle}.

\begin{theorem}
\label{theorem_Ramsey_cycles}
For all integers $r \ge 2$ and $s \ge 2$ we have \[\Ro((C_r,\lhd_{mon}),(C_s,\lhd_{mon}))=2rs-3r-3s+6.\]
\end{theorem}

As a consequence, we obtain tight bounds for geometric and convex geometric Ramsey numbers of cycles introduced by K\'arolyi et al.~\cite{kar97,kar98}; see Corollary~\ref{cor_Karolyi}.

\begin{figure}
\centering
 \includegraphics{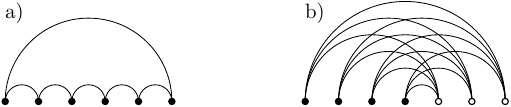} 
\caption{a) The monotone cycle $(C_6,\lhd_{mon})$. b) The ordered complete bipartite graph $\mathcal{K}_{4,3}$.}
\label{fig6_mon_cycle}
\end{figure}

Using a standard probabilistic argument, one can show that there is a constant $c>0$ such that the Ramsey number $\R(G)$ of every graph $G$ with $n$ vertices and $n^{1+\varepsilon}$ edges is at least $2^{cn^{\varepsilon}}$.
On the other hand, it is well-known that if $G$ is an $n$-vertex graph of bounded maximum degree, then $\R(G)$ is linear in $n$~\cite{crst83}.

In a sharp contrast to the latter fact, we construct ordered matchings whose ordered Ramsey numbers grow superpolynomially.

\begin{theorem}
\label{veta_parovani}
There are arbitrarily large ordered matchings $\mathcal{M}$ on $n$ vertices such that 
\[\Ro(\mathcal{M}) \ge n^{\frac{\log{n}}{5\log\log{n}}}.\]
\end{theorem}

Independently, Conlon et al.~\cite{conFox14} showed that a random ordered matching $\mathcal{M}$ on $n$ vertices asymptotically almost surely satisfies $\Ro(\mathcal{M}) \ge n^{\frac{\log{n}}{20\log\log{n}}}$.

In Section~\ref{section_upper_bounds} we give polynomial upper bounds on ordered Ramsey numbers for two classes of sparse ordered graphs. The first class consists of ordered graphs that admit the following decomposition.

For given positive integers $k$ and $q \ge 2$, we say that an ordered graph $\mathcal{G}=(G,\prec)$ is \emph{$(k,q)$-decomposable} if $\mathcal{G}$ has at most $k$ vertices or if it admits the following recursive decomposition: there is a nonempty interval $I\subseteq V(G)$ with at most $k$ vertices such that the interval $I_L$ of vertices of $\mathcal{G}$ that are to the left of $I$ and the interval $I_R$ of vertices of $\mathcal{G}$ that are to the right of $I$ satisfy 
\begin{enumerate}
\item[1)] $|I_L|,|I_R| \le \frac{q-1}{q} \cdot |V(G)|$, 
\item[2)] there is no edge between $I_L$ and $I_R$, and
\item[3)] the induced ordered subgraphs $(G[I_L],\prec \restriction_{I_L})$ and $(G[I_R],\prec \restriction_{I_R})$ are $(k,q)$-decompo\-sable.
\end{enumerate}

\begin{theorem}
\label{veta_rozlozitelne}
Let $k$ and $q\ge 2$ be fixed positive integers.
There is a constant $C'_k$ such that  every $(k,q)$-decomposable ordered graph $\mathcal{G}$ on $n$ vertices satisfies
\[\Ro(\mathcal{G})\le C'_k \cdot n^{128k(q-1)}.\]
\end{theorem}

The constant $C'_k$ depends on $k$ and the proof of Theorem~\ref{veta_rozlozitelne} gives a bound $C'_k \le 2^{O(k\log{k})}$.

We say that the \emph{length of an edge} $uv$ in an ordered graph $(G,\prec)$ is $|i-j|$ if $u$ is the $i$th vertex and $v$ is the $j$th vertex of $G$ in the ordering $\prec$.
The \emph{bandwidth} of $\mathcal{G}$ is the length of the longest edge in $\mathcal{G}$.
Since every ordered graph with bandwidth $k$ is $(k,2)$-decomposable, Theorem~\ref{veta_rozlozitelne} implies the following.

\begin{corollary}
\label{corEdgeLengths}
For every fixed positive integer $k$, there is a constant $C'_k$ such that every $n$-vertex ordered graph $\mathcal{G}$ with bandwidth $k$ satisfies 
\[\Ro(\mathcal{G})\le C'_k \cdot n^{128k}.\] 
\end{corollary}

A graph $G$ is \emph{$k$-degenerate} if there is an ordering $v_1,\dots,v_n$ of its vertices such that every vertex $v_i$ has at most $k$ neighbors $v_j$ in $G$ with $j<i$. 
The \emph{degeneracy} of $G$ is the smallest $k$ such that $G$ is $k$-degenerate.
The degeneracy of an ordered graph $\mathcal{G}=(G,\prec)$ is the degeneracy of the underlying graph $G$.

We give a polynomial upper bound for ordered Ramsey numbers of ordered graphs with constant degeneracy and constant interval chromatic number.
For an ordered graph $\mathcal{G}$, the \emph{interval chromatic number} of $\mathcal{G}$ is the minimum number of intervals the vertex set of $\mathcal{G}$ can be partitioned into such that there is no edge between vertices of the same interval.

\begin{theorem}
\label{thmIntChromDiag}
Every ordered $k$-degenerate graph $\mathcal{G}$ with $n$ vertices and interval chromatic number $p$ satisfies \[\Ro(\mathcal{G}) \le n^{(1+2/k)(k+1)^{\lceil \log{p}\rceil}-2/k}.\]
\end{theorem}

We obtain Theorem~\ref{thmIntChromDiag} in Subsection~\ref{section_proofIntChrom} as a consequence of Theorem~\ref{thmIntChrom}, which is a stronger, asymmetric statement.

Finally, we consider the situation when the given ordered graph $\mathcal{G}$ is fixed and we study the asymptotics of $\Ro(\mathcal{G};c)$ as a function of the number $c$ of colors.
We show that the following dichotomy applies: $\Ro(\mathcal{G};c)$ is either polynomial in $c$ or exponential in~$c$ for every ordered graph~$\mathcal{G}$.
We use $n \cdot \mathcal{K}_{n,n}$ to denote the ordered graph consisting of $n$ disjoint consecutive copies of~$\mathcal{K}_{n,n}$.

\begin{theorem}
\label{thm-dichotomy}
Every ordered graph $\mathcal{G}$ on $n$ vertices satisfies one of the following conditions.
\begin{enumerate}
\item\label{item-dichotomy1} We have $\mathcal{G} \subseteq n\cdot \mathcal{K}_{n,n}$ and $\Ro(\mathcal{G};c) \le (2cn)^{n+1}$.
\item\label{item-dichotomy2} One of the ordered graphs from Figure~\ref{fig_dichotomy} is an ordered subgraph of $\mathcal{G}$ and $\Ro(\mathcal{G};c) > 2^c$.
\end{enumerate}
\end{theorem}

\paragraph{The work by Conlon et al.}
While presenting the results of this paper at the conference Summit 240 in Budapest (2014), we learned about a recent work by Conlon, Fox, Lee, and Sudakov~\cite{conFox14} who independently investigated Ramsey numbers of ordered graphs.
There are some overlaps with our results.

Among many other results, Conlon et al.~\cite{conFox14} proved that as $n$ goes to infinity, almost every ordering $\mathcal{M}_n$ of a matching on $n$ vertices satisfies $\Ro(\mathcal{M}_n)\ge n^{\frac{\log{n}}{20\log\log{n}}}$.
This gives a similar bound as Theorem~\ref{veta_parovani}, where we construct one particular ordered matching on $n$ vertices.

Conlon et al.~\cite{conFox14} also showed that there is a constant $c$ such that every $n$-vertex ordered graph $\mathcal{G}$ with degeneracy $k$ and interval chromatic number $p$ satisfies $\Ro(\mathcal{G})\le n^{ck\log{p}}$. 
This gives a much better estimate than Theorem~\ref{thmIntChromDiag}.

On the other hand, Corollary~\ref{corEdgeLengths} gives a solution to  Problem 6.9 in~\cite{conFox14}.

\subsection{Motivation}
\label{subsection_motivation}

In this subsection we show various examples in which Ramsey-type problems on ordered hypergraphs appear.
The examples consist of both classical and recent results.
Some results and definitions are also used later in the paper.

\paragraph{Erd\H{o}s--Szekeres lemma.}
This well-known fact proved by Erd\H{o}s and Szekeres~\cite{ErSz35_a_comb_problem} states that every sequence of at least $(k-1)^2+1$ distinct integers contains a decreasing or an increasing subsequence of length $k$.
It is easy to see that the bound $(k-1)^2+1$ is sharp.
The Erd\H{o}s--Szekeres lemma has many proofs~\cite{steele95} and it is a special case of a Ramsey-type result for ordered graphs. 

Given a sequence $S=(s_1,\dots,s_N)$ of integers, we construct an ordered graph $(K_N,\prec)$ with vertex set $S$ and the ordering of the vertices given by their positions in $S$.
That is, for $s_i,s_j \in S$, we have $s_i \prec s_j$ if $i<j$.
Then we color an edge $s_is_j$ with $i<j$ red if $s_i < s_j$ and blue otherwise.
Afterwards, red monotone paths correspond to increasing subsequences of $S$ and blue monotone paths to decreasing subsequences of $S$.
The lemma now follows from the following result by Choudum and Ponnusamy~\cite{choudum02} (Milans, Stolee, and West~\cite{milans12} gave a proof in the language of ordered Ramsey theory).

\begin{proposition}[\cite{choudum02}]
\label{prop_mon_paths}
For integers $c,r_1,\dots,r_c \geq 1$,  we have
 \[\Ro((P_{r_1},\lhd_{mon}),\dots\allowbreak,(P_{r_c},\lhd_{mon}))=1+\prod_{i=1}^{c}(r_i-1).\]
\end{proposition}

\paragraph{Erd\H{o}s--Szekeres theorem.} The Erd\H{o}s--Szekeres theorem~\cite{ErSz35_a_comb_problem} was one of the earliest results that contributed to the development of Ramsey theory~\cite{GrNe02_Ramsey_Erdos}.
It states that for every $k \in \mathbb{N}$, there is an integer $\Es(k)$ such that every set of at least $\Es(k)$ points in the plane in \emph{general position} (no three points on a line) contains $k$ points in convex position.

Erd\H{o}s and Szekeres~\cite{ErSz35_a_comb_problem} proved that $\Es(k) \le \binom{2k-4}{k-2}+1$.
Fox, Pach, Sudakov and Suk~\cite{fox12} observed that this bound follows from the equality $\Ro((P^3_k,\lhd_{mon}))= \binom{2k-4}{k-2}+1$.
Moshkovitz and Shapira~\cite{moshkovitz12} discovered a connection between ordered Ramsey numbers of monotone hyperpaths and high-dimensional integer partitions.

\paragraph{Discrete geometry.} 
We arrived at studying Ramsey numbers of ordered graphs while investigating a variant of the Erd\H{o}s--Szekeres problem for so-called 2-page drawings of $K_n$~\cite{chung87}.\footnote{A $2$-page drawing of $K_n$ is a drawing where the vertices of $K_n$ are placed on a horizontal line (the \emph{spine}), the edges between consecutive vertices are subsegments of the spine and each of the remaining edges goes either above or below the spine.
The edges going above the spine are \emph{red} and the edges going below are \emph{blue}.
Let $v_1, \dots, v_n$ be the vertices in the order in which they occur on the spine. 
A triple $v_i$, $v_j$, $v_k$, where $i<j<k$, is colored red if $v_j$ lies below the edge $v_iv_k$ and blue otherwise.
The color of $v_iv_jv_k$ is thus equal to the color of $v_iv_k$. A monochromatic monotone $3$-uniform path corresponds to a monochromatic ($2$-uniform) ordered graph with pairs of vertices $v_i,v_j$ connected if and only if $|i-j|=2$.}
Although this direction of research did not give any interesting results on $2$-page drawings, it lead to Theorem~\ref{veta_rozlozitelne}.

In discrete geometry, \emph{geometric Ramsey numbers}~\cite{cibGao13,kar97, kar98} are natural analogues of ordered Ramsey numbers.
For a finite set of points $P \subset \mathbb{R}^2$ in general position, we denote as $K_P$ the \emph{complete geometric graph on $P$}, which is a complete graph drawn in the plane so that its vertices are represented by the points in $P$ and the edges are drawn as straight-line segments between the pairs of points in $P$.
The graph $K_P$ is \emph{convex} if $P$ is in convex position.
The \emph{geometric Ramsey number} of a graph $G$, denoted by $\Rg(G)$, is the smallest $N$ such that every complete geometric graph $K_P$ on $N$ vertices with edges colored by two colors contains a noncrossing monochromatic drawing of $G$.
If we consider only convex complete geometric graphs $K_P$ in the definition, then we get so-called \emph{convex geometric Ramsey number} $\Rc(G)$.
Note that these numbers are finite only if $G$ is outerplanar and that $\Rc(G) \le \Rg(G)$ for every outerplanar graph $G$.

For the cycles $C_n$ with $n \ge 3$, K\'arolyi et al.~\cite{kar98} showed the upper bound $\Rg(C_n) \le 2n^2-6n+6$ and also observed that $\Rc(C_n) \ge (n-1)^2+1$.

Using Theorem~\ref{theorem_Ramsey_cycles}, we show that the geometric and convex geometric Ramsey numbers of cycles are equal to the ordered Ramsey numbers of monotone cycles; see Corollary~\ref{cor_Karolyi}.
First we observe that the ordered and convex geometric Ramsey numbers of cycles are the same.

\begin{observation}
\label{obsDiscreteGeo}
For every $n \ge 3$, we have $\Rc(C_n) = \Ro((C_n,\lhd_{mon}))$.
\end{observation}

\begin{proof}
Consider a set of $n$ points in convex position.
Order the points $v_1 \prec \dots \prec v_n$ in the clockwise order starting at an arbitrary vertex.
The observation follows from the fact that a cycle with vertex set $\{v_1, \dots, v_n\}$ is noncrossing if and only if it is the monotone cycle with respect to $\prec$.
\end{proof}

We also sketch a connection between ordered Ramsey numbers and convex geometric Ramsey numbers of outerplanar graphs in Section~\ref{section_open_problems}.

\paragraph{Extremal problems on matrices.}
Extremal theory of $\{0,1\}$-matrices~\cite{cib13,fur92} is essentially concerned with ordered Tur\'an numbers of ordered bipartite graphs.
We use some results from this theory in Subsection~\ref{subsection_paths}. 

A $\{0,1\}$-matrix $A$ \emph{contains} an $r \times s$ submatrix $M$ if $A$ contains a submatrix $B$ that has $1$-entries at all the positions where $M$ does.
A matrix $A$ \emph{avoids} $M$ if it does not contain~$M$.
The \emph{extremal function of $M$} is the maximum number $\ex_{M}(m,n)$ of $1$-entries in an $m \times n$ $\{0,1\}$-matrix avoiding $M$.

Let $K_{n_1,n_2}$ be the complete bipartite graph with parts of sizes $n_1$ and $n_2$.
We use $\mathcal{K}_{n_1,n_2}$ to denote the ordering of $K_{n_1,n_2}$ in which the two parts form disjoint intervals such that the interval of size $n_1$ is to the left of the interval of size $n_2$; see part~b) of Figure~\ref{fig6_mon_cycle}.

Let $\mathcal{G}$ and $\mathcal{H}$ be ordered graphs. The \emph{ordered Tur\'an number of $\mathcal{G}$ in $\mathcal{H}$} is the maximum number of edges in an ordered subgraph $\mathcal{H'}$ of $\mathcal{H}$ such that $\mathcal{G}$ is not an ordered subgraph of~$\mathcal{H'}$.

Let $\mathcal{G}=((A \cup B,E),\prec)$, with $|A|=r$ and $|B|=s$, be a a subgraph of $\mathcal{K}_{r,s}$. Then $\mathcal{G}$ corresponds to an $r\times s$ $\{0,1\}$-matrix $M(\mathcal{G})$ where $M(\mathcal{G})_{i,j}=1$ if the $i$th vertex in $A$ and the $j$th vertex in $B$ are adjacent in $\mathcal{G}$, and $M(\mathcal{G})_{i,j}=0$ otherwise.
It is easy to see that the ordered Tur\'an number of $\mathcal{G}$ in $\mathcal{K}_{m,n}$ equals $\ex_{M(\mathcal{G})}(m,n)$.


\section{Ordered Ramsey numbers for specific classes of graphs}\label{section_specific}

In this section we compute ordered Ramsey numbers of certain ordered stars, paths and cycles. 
We compare the formulas and bounds obtained with known Ramsey numbers of the corresponding unordered graphs.


\subsection{Ordered stars}\label{subsection_stars}

A \emph{star} with $n$ vertices is the complete bipartite graph $K_{1,n-1}$.
Ramsey numbers of unordered stars are known exactly~\cite{burr73} and they are given by 
\[\R(K_{1,n-1};c)=\begin{cases}c(n-2)+1 & \text{if } c \equiv n-1 \equiv 0\; (\bmod \; 2), \\c(n-2)+2 & \text{otherwise}.  \end{cases}\]

The position of the central vertex of an ordered star determines the ordering of the star uniquely up to isomorphism.
We use $\mathcal{S}_{r,s}$ to denote the ordered star with $r-1$ vertices to the left and $s-1$ vertices to the right of the central vertex; see Figure~\ref{fig15_stars}.

For $c,r_1,\dots,r_c \ge 2$, computing $\Ro(\mathcal{S}_{1,r_1},\dots,\mathcal{S}_{1,r_c})$ is straightforward.
In the diagonal case, the ordered Ramsey numbers $\Ro(\mathcal{S}_{1,n};c)$ are equal to the Ramsey numbers $\R(K_{1,n-1};c)$ for every $n$ and $c$, if $n$ is even or $c$ is odd.

\begin{figure}
\centering
\includegraphics{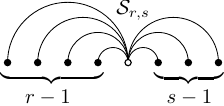} 
\caption{The ordered star $\mathcal{S}_{r,s}$.}
\label{fig15_stars}
\end{figure}

\begin{observation}
\label{obsStars}
For all integers $c,r_1,\dots,r_c \ge 2$ we have 
\[
\Ro(\mathcal{S}_{1,r_1},\dots,\mathcal{S}_{1,r_c}) = 2 + \sum_{i=1}^{c}(r_i-2).
\]
\end{observation}

\begin{proof}
Let $\mathcal{K}_N$ be an ordered complete graph with $N \ge 2 + \sum_{i=1}^{c}(r_i-2)$ vertices and edges colored with $c$ colors.
By the pigeonhole principle, for some $i\in [c]$, the leftmost vertex in $\mathcal{K}_N$ has at least $r_i-1$ incident edges of color $i$.
These edges form a copy of $\mathcal{S}_{1,r_i}$.

The following $c$-coloring of the edges of $\mathcal{K}_N$ with $N\mathrel{\mathop:}=1+\sum_{i=1}^{c}(r_i-2)$ has no star $\mathcal{S}_{1,r_i}$ in color $i$.
Partition all the vertices of $\mathcal{K}_N$ except for the leftmost vertex into $c$ subsets $V_1,\dots, V_c$ such that $|V_i|= r_i-2$.
Then color each edge with its right vertex in $V_i$ by color $i$.
Thus no color $i$ contains a copy of $\mathcal{S}_{1,r_i}$, since otherwise all $r_i-1$ right neighbors of the central vertex of $\mathcal{S}_{1,r_i}$ are in $V_i$, which is impossible.
\end{proof}

Choudum and Ponnusamy~\cite {choudum02} determined the ordered Ramsey numbers of all pairs of ordered stars by the following recursive formulas.

\begin{theorem}[\cite{choudum02}]
\label{thmStarsPair}
For all integers $r_1,r_2 > 2$, we have
\[\Ro(\mathcal{S}_{1,r_1},\mathcal{S}_{r_2,1})=\left\lfloor \frac{-1+\sqrt{1+8(r_1-2)(r_2-2)}}{2} \right\rfloor+r_1+r_2-2.\]
Moreover, for all integers $r_1,r_2,s_1,s_2 \ge 2$, we have
\[\Ro(\mathcal{S}_{1,r_1},\mathcal{S}_{r_2,s_2})=\Ro(\mathcal{S}_{1,r_1},\mathcal{S}_{r_2,1})+r_1+s_2-3\]
and
\[\Ro(\mathcal{S}_{r_1,s_1},\mathcal{S}_{r_2,s_2})=\Ro(\mathcal{S}_{r_1,1},\mathcal{S}_{r_2,s_2})+\Ro(\mathcal{S}_{1,s_1},\mathcal{S}_{r_2,s_2})-1.\]
\end{theorem}

\subsubsection*{General upper bound for ordered stars}

\begin{proof}[{Proof of Theorem~\ref{thmStars}}]
For every $i\in [c]$, let $r_i$ and $s_i$ be positive integers such that $\mathcal{S}_i=\mathcal{S}_{r_i,s_i}$.
Let $N$ be a positive integer.
Suppose that the edges of $\mathcal{K}_N$ are colored by colors from $[c]$ so that for every $i\in [c]$, there is no $\mathcal{S}_{r_i,s_i}$ in color $i$.
Thus, in the ordered subgraph $\mathcal{G}_i$ of $\mathcal{K}_N$ formed by the edges of color $i$, every vertex has at most $r_i-2$ left neighbors or at most $s_i-2$ right neighbors.
Let $\mathcal{H}_i$ be the ordered subgraph of $\mathcal{G}_i$ obtained by deleting every edge incident from the left to a vertex with at most $r_i-2$ left neighbors, and every edge incident from the right to a vertex with at most $s_i-2$ right neighbors.
Clearly, $|E(\mathcal{G}_i) \setminus E(\mathcal{H}_i)| \le N \cdot (r_i+s_i-4) = N \cdot (n_i-3)$.
It follows that the ordered graph $\mathcal{H} \mathrel{\mathop:}= \bigcup_{i=1}^{c} \mathcal{H}_i$ has at least $\binom{N}{2} - N \cdot \sum_{i=1}^{c}(n_i-3) = N \cdot \left(N / 2 - 1/2 - \sum_{i=1}^{c}(n_i-3)\right)$ edges.

By the construction, each of the ordered graphs $\mathcal{H}_i$ is bipartite.
Hence, the ordered graph $\mathcal{H}$ is $2^c$-partite (in other words, $2^c$-colorable).
Therefore, by the AM--GM inequality or by Tur\'{a}n's theorem, $|E(\mathcal{H})| \le (1-1/2^c)\cdot N^2/2 = N \cdot (N/2-N/2^{c+1})$.
Putting the two estimates together, we obtain that $N/2^{c+1} \le 1/2 + \sum_{i=1}^{c}(n_i-3)$, from which the theorem follows.
\end{proof}

\subsubsection*{Lower bound for ordered stars with interval chromatic number $3$}

We give a lower bound for ordered Ramsey numbers of ordered stars that have at least one edge incident to the central vertex from each side.
For ``symmetric'' stars $\mathcal{S}_{r_i,r_i}$ with $r_i\ge 2$, the lower bound is within a constant multiplicative factor from the upper bound in Theorem~\ref{thmStars}.
For $r_1=\dots=r_c=s_1=\dots=s_c=2$, the bound is exactly the same as in Proposition~\ref{prop_mon_paths}.

\begin{proposition}
\label{prop_star_lower}
For all integers $c\ge 2$ and $r_1,\dots,r_c,s_1,\dots,s_c\ge 2$, we have 
\[
\Ro(\mathcal{S}_{r_1,s_1},\dots,\mathcal{S}_{r_c,s_c})>2^{c-1} \cdot \max\bigg(\max_{i\in[c]}\{r_i+s_i-2\}, 2+2\cdot\sum_{i=1}^{c} (\min(r_i, s_i) - 2)\bigg).
\]
\end{proposition}

\begin{proof}
Let $a \mathrel{\mathop:}=\max_{i\in[c]}\{r_i+s_i-2\}$, $b \mathrel{\mathop:}= 1+\sum_{i=1}^{c} (\min(r_i, s_i) - 2)$, and $N_1 \mathrel{\mathop:}= \max(a,2b)$.
Without loss of generality, we assume that $a=r_1+s_1-2$.
We construct $c$-colorings of the complete ordered graphs with $2^{i-1}\cdot N_1$ vertices, for $i=1,2,\dots,c$, by induction on $i$.

We start the construction with coloring the edges of $\mathcal{K}_{N_1}$. If $N_1=a$, we color every edge of $\mathcal{K}_{N_1}$ by color $1$.
Now suppose that $N_1=2b$.
For every $i\in [c]$, let $t_i \mathrel{\mathop:}= \min(r_i, s_i) - 2$, and let $b_i$ be the partial sum $\sum_{j=1}^{i} t_j$.
In particular, $b_c=b-1$.
Let $b_0 \mathrel{\mathop:}= 0$ and let $v_1, v_2, \dots, v_{N_1}$ be the vertices of $\mathcal{K}_{N_1}$ from left to right.
For every $k,l \in [b]$, $k<l$, color the edge $v_kv_l$ by color $i$ if $b_{i-1}<l-k\le b_i$.
In this coloring of the subgraph $\mathcal{K}_b=\mathcal{K}_{N_1}[v_1,\dots,v_b]$, every vertex has at most $t_i$ left neighbors and at most $t_i$ right neighbors joined by an edge of color $i$.
Color the subgraph $\mathcal{K}_{N_1}[v_{b+1},\dots,v_{2b}]$ analogously as $\mathcal{K}_b$, and finally, color every edge $v_iv_j$ with $i\le b<j$ by color $1$.

For $i\in \{2,3,\dots,c\}$, let $N_i \mathrel{\mathop:}= 2^{i-1}\cdot N_1$.
Once $\mathcal{K}_{N_i}$ is colored, we split the vertices $v_1,\dots,v_{N_{i+1}}$ of $\mathcal{K}_{N_{i+1}}$ into two intervals of length $N_i$, and color the subgraph induced by each of the two intervals using the coloring of $\mathcal{K}_{N_i}$.
Then we color every edge between the two intervals by color $i$.
See Figure~\ref{fig_coloring}.

It remains to verify that there is no monochromatic copy of $\mathcal{S}_{r_i,s_i}$ in the resulting coloring of $\mathcal{K}_{N_c}$.
In the case $N_1 = a$, every vertex has at most $r_1+s_1-3$ neighbors joined by an edge of color $1$ and for every other $i \in [c]$, it either has no left or no right neighbors in color $i$.
In the case $N_1=2b$, for every $i\in [c]$, every vertex has at most $r_i-2$ left neighbors or at most $s_i-2$ right neighbors in color $i$. 
\end{proof}

\begin{figure}
\centering
 \includegraphics{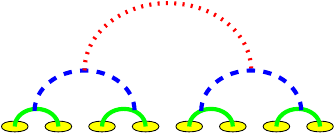} 
\caption{An inductive construction of the coloring avoiding stars $\mathcal{S}_{r_i,s_i}$ with $r_i,s_i\ge 2$.}
\label{fig_coloring}
\end{figure}


\subsection{Ordered paths}\label{subsection_paths}

Gerencs\'er and Gy\'arf\'as~\cite{gerencser67} determined the exact values for the Ramsey numbers $\R(P_r,P_s)$ of two paths $P_r$ and $P_s$. 

\begin{theorem}[{\cite{gerencser67}}]
For $2\le r \le s$, we have $\R(P_r,P_s)=s+\left\lfloor \frac{r}{2}\right\rfloor-1$.
\end{theorem} 

Here we prove Proposition~\ref{prop_alt_paths}, which shows that ordered Ramsey numbers for a particular ordering scheme of paths are linear in the number of vertices.
This result contrasts with Proposition~\ref{prop_mon_paths} and even more strongly with Theorem~\ref{veta_parovani}.
Let $v_1,\dots,v_n$ be the vertices of $P_n$ in the order as they appear along the path.
The \emph{alternating path} $(P_n,\lhd_{alt})$ is an ordered path where $v_1 \lhd_{alt} v_3 \lhd_{alt} v_5 \lhd_{alt} \dots \lhd_{alt} v_n \lhd_{alt} v_{n-1} \lhd_{alt} v_{n-3} \lhd_{alt} \dots \lhd_{alt} v_2$ for $n$ odd and $v_1 \lhd_{alt} v_3 \lhd_{alt} v_5 \lhd_{alt} \dots \lhd_{alt} v_{n-1} \lhd_{alt} v_{n} \lhd_{alt} v_{n-2} \lhd_{alt}\dots \lhd_{alt} v_2$ for $n$ even.
See Figure~\ref{fig_alt_path}.
Obviously, the alternating path $(P_n,\lhd_{alt})$ is an ordered subgraph of $\mathcal{K}_{\lceil n/2 \rceil,\lfloor n/2 \rfloor}$, and so it has interval chromatic number $2$. 

\begin{figure}
\centering
 \includegraphics{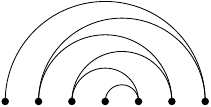} 
\caption{The alternating path $(P_7,\lhd_{alt})$.}
\label{fig_alt_path}
\end{figure}

To prove Proposition~\ref{prop_alt_paths}, we use a result from extremal theory of $\{0,1\}$-matrices (see Subsection~\ref{subsection_motivation} for definitions).
The following definitions are taken from~\cite{cib13}.
We say that an $r \times s$ matrix $M$ is \emph{minimalist} if $\ex_M(m,n)=(s-1)m+(r-1)n-(r-1)(s-1)$ when $m\ge r$ and $n\ge s$.
Clearly, the $1 \times 1$ identity matrix is minimalist.
If the matrix $M^\prime$ was created from a matrix $M$ by adding a new row (or a column) as the new first or last row (column) and this new row (column) contains a single 1-entry next to a 1-entry of $M$, then we say that $M^\prime$ was created by an \emph{elementary operation} from $M$. F\"{u}redi and Hajnal~\cite{fur92} proved the following lemma.

\begin{lemma}[\cite{fur92}]
Let $M$ be an $r \times s$ minimalist matrix and let $M^\prime$ be an $r^\prime \times s^\prime$ matrix obtained from $M$ by applying a sequence of elementary operations.
Then $M^\prime$ is minimalist.
\label{lemmaFuredi}
\end{lemma}

\begin{proof}[Proof of Proposition~\ref{prop_alt_paths}]
We show that, for $n>2$, we have $5\lfloor n/2 \rfloor -4 \le \Ro((P_n,\lhd_{alt})) \le 2n-3+\sqrt{2n^2-8n+11}$.
For the lower bound, let $r\mathrel{\mathop:}=\lfloor n/2 \rfloor - 1$ and $N \mathrel{\mathop:}=5r$.
We use an upper triangular $\{0,1\}$-matrix $A = (a_{i,j})_{i,j=1}^N$ to represent a red-blue coloring $c$ of $\mathcal{K}_N$ that avoids $(P_n,\lhd_{alt})$.
The construction of $A$ is sketched in part~a) of Figure~\ref{fig_alt_path_proof}.
For $1 \leq i<j \le N$, the edge $ij$ of $\mathcal{K}_N$ is blue in $c$ if $A_{i,j}=0$ and red in $c$ if $A_{i,j}=1$.
For integers $k$ and $l$ with $1 \le k \le l\le 5$, we use $B_{k,l}$ to denote the $r\times r$ blocks that partition~$A$.
The block $B_{k,l}$ contains entries $a_{i,j}$ with $(k-1)r+1\le i \leq kr$ and $(l-1)r+1\le j \leq lr$.
There are two types of blocks.
\emph{Red blocks}, containing only 1-entries, and \emph{blue blocks}, containing only 0-entries.
The blocks $B_{1,5}, B_{2,2}, B_{2,3},B_{3,3},B_{3,4},B_{4,4}$ are red and all the other blocks are blue.

Suppose for contradiction that $c$ contains a monochromatic copy $\mathcal{P}$ of $(P_n,\lhd_{alt})$.
Then there are entries $a_{i_1,j_1}=\dots=a_{i_{n-1},j_{n-1}}$ in~$A$ such that, for $t=1,\ldots,n-2$, we have $i_t<i_{t+1}$ and $j_t=j_{t+1}$ if $t$ is odd and $i_t=i_{t+1}$ and $j_t>j_{t+1}$ if $t$ is even.
Every block $B_{k,l}$ has at most $r$ rows and at most $r$ columns of $A$, but $\mathcal{P}$ spans at least $r+1$ rows and $r+1$ columns.
Therefore there are three blocks $B_1,B_2,B_3$ of the same color that satisfy one of the following two conditions.
Either the block $B_1$ is above $B_2$ and $B_3$ is to the left of $B_2$ or the block $B_1$ is to the right of $B_2$ and $B_3$ is below $B_2$.
However, the matrix $A$ contains no such triples of blocks, a contradiction.

Several alternative constructions of $A$ are illustrated in parts b), c), and d) of Figure~\ref{fig_alt_path_proof}.

\begin{figure}
\centering
 \includegraphics{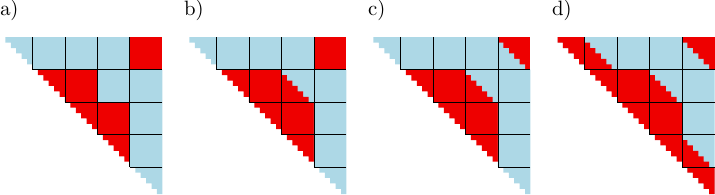} 
\caption{a)--d) Various constructions for the lower bound on $\Ro((P_n, \lhd_{alt}))$.
The red entries represent 1, the blue entries represent 0. 
}
\label{fig_alt_path_proof}
\end{figure}

Now we show the upper bound.
Let $N$ be a positive integer.
Suppose that the edges of $\mathcal{K}_{\lceil N/2\rceil,\lfloor N/2\rfloor}$ are colored red and blue.
Without loss of generality, at least half of the edges are red.
Consider the $\lceil n/2\rceil \times \lfloor n/2\rfloor$ $\{0,1\}$-matrix $M \mathrel{\mathop:}= M((P_n,\lhd_{alt}))$, which is defined in Subsection~\ref{subsection_motivation}.
By Lemma~\ref{lemmaFuredi}, $M$ is minimalist, as $M$ is obtained from the $1 \times 1$ matrix with a single $1$-entry by alternately adding a new last row with a single $1$-entry just below another $1$-entry and a new first column with a single $1$-entry to the left of another $1$-entry. Therefore, 
\begin{align*}
\ex_M(\lceil N/2\rceil,\lfloor N/2\rfloor) &= (\lfloor n/2\rfloor-1)\lceil N/2\rceil+(\lceil n/2\rceil-1)\lfloor N/2\rfloor\\
&-(\lceil n/2\rceil-1)(\lfloor n/2\rfloor-1)
\le \frac{1}{4}(2nN+4n-4N-3-n^2).
\end{align*} 
The subgraph of $\mathcal{K}_{\lceil N/2\rceil,\lfloor N/2\rfloor}$ formed by red edges has at least $\frac{1}{2}\lceil N/2\rceil \cdot \lfloor N/2\rfloor \ge (N^2-1)/8$ edges.
Hence, if there is no red copy of $(P_n,\lhd_{alt})$, the inequality
\[
N^2-1 \le 4nN+8n-8N-6-2n^2 
\]
is satisfied.
Consequently, $N \le 2n-4+\sqrt{2n^2-8n+11}$ and the result follows.
\end{proof}

We do not know the precise multiplicative factor in $\Ro((P_n,\lhd_{alt}))$.
Our computer experiments~\cite{bckkWeb} indicate that $\Ro(P_n,\lhd_{alt})$ could be of the form $\lfloor (n-2)\frac{1+\sqrt{5}}{2}\rfloor+n$; see Table~\ref{tab_alt_path}.
In our experiments we used the Glucose SAT solver~\cite{aud12}.

\begin{table}[ht]
\begin{center} 
 \renewcommand{\arraystretch}{1.2}

 \begin{tabular}{c|r|r|r|r|r|r|r|r|r|r|r|r}
  $n$ & $2$ & $3$ & $4$ & $5$ & $6$ & $7$ & $8$ & $9$ & $10$ & $11$ & $12$ & $13$ \\
\hline
  $\Ro(n)$ & $2$ & $4$ & $7$ & $9$ & $12$ & $15$ & $17$ & $\ge 20$ & $\ge 22$ & $\ge 25$ & $\ge 28$ & $\ge 30$\\
 \end{tabular}
 \caption{Estimates of the ordered Ramsey numbers $\Ro(n) \mathrel{\mathop:}= \Ro((P_n,\lhd_{alt}))$ for $n \le 13$.}
 \label{tab_alt_path}
\end{center}
\end{table}

For general ordered paths, Cibulka et al.~\cite{cibGao13} showed that  for every ordered path $\mathcal{P}_r$ and the ordered complete graph $\mathcal{K}_s$ we have 
\[
\Ro(\mathcal{P}_r,\mathcal{K}_s) \le 2^{\lceil \log s\rceil (\lceil \log r \rceil + 1)}.
\] 
That is, for every ordered path $\mathcal{P}_n$ we have $\Ro(\mathcal{P}_n) \le n^{O(\log n)}$. 
This bound also follows from a result of Conlon et al.~\cite{conFox14} who showed that there is a constant $c$ such that every $n$-vertex ordered graph $\mathcal{G}$ with degeneracy $k$ and interval chromatic number $p$ satisfies $\Ro(\mathcal{G})\le n^{ck\log{p}}$. 

The best known lower bound on $\Ro(\mathcal{P}_n)$ comes from Theorem~\ref{veta_parovani}, which implies that there are arbitrarily 
long ordered paths $\mathcal{P}_n$ on $n$ vertices such that $\Ro(\mathcal{P}_n) \ge n^{\log{n}/ (5\log\log{n})}.$


\subsection{Ordered cycles}\label{subsection_cycles}

It is a folklore result in Ramsey theory that $\R(C_3)=\R(C_4)=6$~\cite{chvat72}.
The first results on Ramsey numbers of cycles were obtained by Chartrand and Chuster~\cite{chartrand71} and by Bondy and Erd\H{o}s~\cite{bondy73}.
These were later extended by Rosta~\cite{rosta73} and by Faudree and Schelp~\cite{faudree74}.
Together, these results give exact formulas for Ramsey numbers of cycles in the two-color case.
\[
\R(C_r,C_s)=
\begin{cases}
2r-1 & \textrm{if } (r,s)\neq (3,3), r \ge s \ge 3, s \textrm{ is odd}, 
\\r+s/2-1 & \textrm{if } (r,s)\neq (4,4), r \ge s \ge 4, r, s \textrm{ are even},
\\\max\{r+s/2,2s\}-1 & \textrm{if } r > s \ge 4, s \textrm{ is even and } r \textrm{ is odd}.  
\end{cases}
\]

The smallest cycle whose ordered Ramsey numbers are nontrivial to determine is $C_4$. There are three pairwise nonisomorphic orderings of $C_4$; see Figure~\ref{fig4_C4}.
We determine the ordered Ramsey number of each of the three orderings of $C_4$.

\begin{figure}
\centering
 \includegraphics{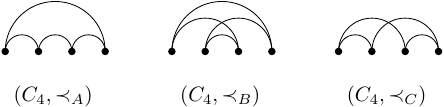} 
\caption{Three possible orderings of $C_4$.}
\label{fig4_C4}
\end{figure}

\begin{proposition}
\label{propC4}
We have \begin{enumerate}
\item[{\rm 1)}] $\Ro((C_4,\prec_A)) = 14$, 
\item[{\rm 2)}] $\Ro((C_4,\prec_B)) =10$, 
\item[{\rm 3)}] $\Ro((C_4,\prec_C)) =11$.
\end{enumerate}
\end{proposition}

\begin{proof}
We provide colorings showing the lower bounds on a separate webpage~\cite{bckkWeb}.
We now show the upper bounds. 

\begin{enumerate}
\item[{\rm 1)}] This result follows from Theorem~\ref{theorem_Ramsey_cycles}, which is proved later in this subsection.

\item[{\rm 2)}]  Consider $(K_{10},\prec)$ with vertices $v_1 \prec v_2 \prec \dots \prec v_{10}$ and edges colored red and blue.
We put each of the vertices $\{v_4, v_5, \dots, v_{10}\}$ into one of the following six classes.
Class $(i,c)$, where $i \in \{1,2,3\}$ and $c \in\{\text{red}, \text{blue}\}$, contains vertices connected by an edge of color $c$ to both vertices in $\{v_1, v_2, v_3\} \setminus \{v_i\}$.
Note that each of the seven vertices $\{v_4, v_5, \dots, v_{10}\}$ is in one or three of these classes.
Consequently, one of the six classes contains at least two vertices.
Thus there are two vertices that share two left neighbors of the same color, which implies a monochromatic copy of $(C_4,\prec_B)$.

\item[{\rm 3)}] 
Exhaustive computer search showed that $\Ro((C_4,\prec_C))=11$~\cite{bckkWeb}. Here we prove a weaker upper bound $\Ro((C_4,\prec_C))\le 13$.

Consider $(K_{13},\prec)$ with vertices $v_1 \prec \dots \prec v_{13}$, edges colored red and blue, and no monochromatic $(C_4,\prec_C)$.
Without loss of generality, $v_1$ has six red neighbors among $\{v_2,v_3,\dots,v_{12}\}$.
If $v_1$ and $v_{13}$ had two common red neighbors then they would form a red copy of $(C_4,\prec_C)$.
Thus there is a set $R \subseteq \{v_2,v_3,\dots,v_{12}\}$ of at least five vertices such that each of them is adjacent to $v_1$ by a red edge and to $v_{13}$ by a blue edge.
By Theorem~\ref{thmStarsPair} we have $\Ro(\mathcal{S}_{1,3},\mathcal{S}_{3,1})=5$.
Therefore the complete graph formed by the five vertices of $R$ contains either a vertex with at least two red edges incident from the left or a vertex with at least two blue edges incident from the right.
In both cases we obtain a  monochromatic  copy of $(C_4,\prec_C)$. \qedhere

\end{enumerate}
\end{proof}


\subsubsection*{Monotone cycles}
Here we prove Theorem~\ref{theorem_Ramsey_cycles}. We use the following simple lemma, which is implicitly proved in~\cite{kar97}. We include its proof for completeness.

\begin{lemma}[{\cite{kar97}}]
\label{lemCycC3Path}
For positive integers $r$ and $s$, we have \[\Ro((P_r,\lhd_{mon}),\mathcal{K}_s)=\Ro((P_r,\lhd_{mon}),(P_s,\lhd_{mon}))=(r-1)(s-1)+1.\]
\end{lemma}
\begin{proof}
The lower bound $(r-1)(s-1)+1\le\Ro((P_r,\lhd_{mon}),(P_s,\lhd_{mon}))$ follows from Proposition~\ref{prop_mon_paths}.
For the upper bound $\Ro((P_r,\lhd_{mon}),\mathcal{K}_s)\le(r-1)(s-1)+1$, we apply induction on $r$.
Let $\mathcal{G}$ be an ordered complete graph with $(r-1)(s-1)+1$ vertices and with edges colored red and blue. The statement is true for $r=2$, since either $\mathcal{G}$ is a blue copy of $\mathcal{K}_s$ or $\mathcal{G}$ has a red edge.
Let $r\ge 3$. By the induction hypothesis, $\mathcal{G}$ has either a blue copy of $\mathcal{K}_s$ or at least $(r-1)(s-1)+1-(r-2)(s-1)=s$ distinct vertices that are the rightmost vertices of a red copy of $(P_{r-1},\lhd_{mon})$. Either every edge between these vertices is blue, which gives a blue copy of $\mathcal{K}_s$, or a red edge extends one of the red paths $(P_{r-1},\lhd_{mon})$ to a red copy of  $(P_{r},\lhd_{mon})$.
\end{proof}

\paragraph{Proof of Theorem~\ref{theorem_Ramsey_cycles}.}
The upper bound was proved by K\'arolyi et al.~\cite[Theorem 2.1]{kar98}.
We include the proof here for completeness.

Let $\mathcal{G}$ be an ordered complete graph with $N \mathrel{\mathop:}= 2rs-3r-3s+6$ vertices and with edges colored red and blue. The leftmost vertex, $v_1$, has either at least $(r-2)(s-1)+1$ red neighbors or at least $(r-1)(s-2)+1$ blue neighbors.
In the first case, by Lemma~\ref{lemCycC3Path}, $\mathcal{G}$ has a red copy of $(P_{r-1},\lhd_{mon})$ that forms a red copy of $(C_r,\lhd_{mon})$ together with $v_1$, or a blue copy of $(C_s,\lhd_{mon})$.
The second case is symmetric.

\begin{figure}
\centering
 \includegraphics{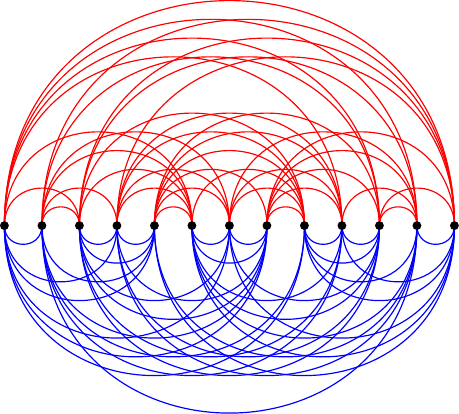} 
\caption{A coloring of $\mathcal{K}_{13}$ with no monochromatic monotone cycle of length $4$.}
\label{fig5_C4_lower_bound}
\end{figure}

Now we prove the lower bound. Let $N \mathrel{\mathop:}= 2rs-3r-3s+5$. We construct a coloring of $\mathcal{K}_N=(K_N,\prec)$ that avoids a red copy of $(C_r,\lhd_{mon})$ and a blue copy of $(C_s,\lhd_{mon})$.
See Figure~\ref{fig5_C4_lower_bound} for an example of such coloring for $r=s=4$.
We partition the vertices of $\mathcal{K}_N$ into disjoint intervals $I_1,\dots,I_{2r-3}$, from left to right.
For $r$ odd, the $(r-1)/2$ leftmost and $(r-1)/2$ rightmost intervals are of size $s-1$ and the remaining $r-2$ intervals are of size $s-2$. 
For $r$ even, the $(r-2)/2$ leftmost and $(r-2)/2$ rightmost intervals are of size $s-2$ and the remaining $r-1$ intervals are of size $s-1$.
In both cases we have $N$ vertices in total.

We call the intervals of size $s-1$ \emph{long} and the intervals of size $s-2$ \emph{short}.
We label the vertices of $I_i$ as $v^i_1, v^i_2, \dots, v^i_{\lvert I_i \rvert}$ from left to right.
We call the index $j$ the \emph{index of the vertex} $v^i_j$.

The coloring of the edges of $\mathcal{K}_N$ is defined as follows.
For every $i\in [2r-3]$, we color all the edges among the vertices of $I_i$ blue.
We define four types of edges with vertices in different intervals. The type of an edge is determined by the pair of intervals containing its vertices. The color of an edge is determined by its type and by the relative value of the indices of its vertices. We say that an edge $e=v^i_kv^j_l$ between intervals $I_i$ and $I_j$ with $i<j$ is of \emph{type}
\begin{itemize}
\item $T_<$ if $j-i \le r-2$ and $|I_i| \le |I_j|$. In this case we color $e$ blue if $k<l$ and red otherwise.
\item $T_{\ge}$ if $j-i > r-2$ and $|I_i| < |I_j|$. In this case we color $e$ blue if $k\ge l$ and red otherwise.
\item $T_{>}$ if $j-i > r-2$ and $|I_i| \ge |I_j|$. In this case we color $e$ blue if $k> l$ and red otherwise.
\item $T_{\le}$ if $j-i \le r-2$ and $|I_i| > |I_j|$. In this case we color $e$ blue if $k\le l$ and red otherwise.
\end{itemize}

The definition of the types and the distribution of the types in $\mathcal{K}_N$ are illustrated in Figures~\ref{fig7_cycle_types} and~\ref{fig13_cycle_types_distr}, respectively.

The distribution of long and short intervals implies the following claim.

\begin{claim}
\label{claim_there_can_be_only_one}
Every monochromatic monotone path $\mathcal{P}$ in the constructed coloring of $\mathcal{K}_N$ contains at most one edge $uv$, $u \prec v$, with $u$ in a long interval and $v$ in a short interval, and at most
one edge $u'v'$, $u'\prec v'$, with $u'$ in a short interval and $v'$ in a long interval.
\end{claim}
\begin{proof}
Let $uv$, where $u \prec v$, be the first edge of $\mathcal{P}$ with $u$ in a long interval and $v$ in a short interval.
If $r$ is odd, then $u$ lies in $J_r = \bigcup_{i=1}^{(r-1)/2} I_i$ and $v$ lies to the right of $J_r$.
If $r$ is even, then $u \in J_r = \bigcup_{i=(r-2)/2+1}^{(5r-4)/2} I_i$ and $v$ is to the right of $J_r$.
Any other edge of $\mathcal{P}$ with the left vertex in a long interval and with the right one in a short interval would have the left vertex in $J_r$ and also to the right of $v$.
However, this is impossible, as $v$ is to the right of $J_r$.
The argument for the edge $u'v'$ is analogous.
\end{proof}

\begin{figure}
\centering
 \includegraphics{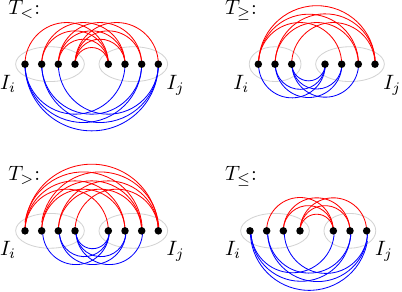} 
\caption{The types of pairs $(I_i,I_j)$ for $s=5$ and colorings of corresponding edges.}
\label{fig7_cycle_types}
\end{figure}

\begin{figure}
\centering
 \includegraphics{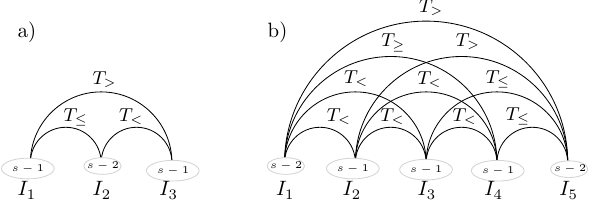} 
\caption{Distribution of types of pairs $(I_i,I_j)$ in $\mathcal{K}_N$ for a) $r=3$ and b) $r=4$.}
\label{fig13_cycle_types_distr}
\end{figure}

First we show that our coloring of $\mathcal{K}_N$ contains no red copy of $(C_r,\lhd_{mon})$.
Suppose for contradiction that there is such a copy $\mathcal{C}$. Let $\mathcal{P}$ be the monotone path on $r$ vertices contained in $\mathcal{C}$. Let $u$ be the leftmost vertex of $\mathcal{P}$ and $v$ the rightmost vertex of $\mathcal{P}$. The edge $uv$ is thus the longest edge of $\mathcal{C}$.
Note that $\mathcal{C}$ contains at most one vertex from each interval $I_i$, as every interval contains only blue edges.
The path $\mathcal{P}$ contains no edge of type $T_{>}$ or $T_{\ge}$, since otherwise $\mathcal{P}$ would skip vertices from at least $r-2$ intervals, leaving at most $2r-3-(r-2)=r-1$ intervals.
Hence the vertex indices in $\mathcal{P}$ are nonincreasing from left to right, as $\mathcal{P}$ uses red edges of types $T_{<}$ and $T_{\le}$ only. 

Since the edge $uv$ skips at least $r-2$ intervals, it is of type $T_{>}$ or $T_{\ge}$, and thus the index of $v$ is at least as large as the index of $u$. In combination with the previous observation, this implies that the indices of $u$ and $v$ are equal. Consequently, $uv$ is of type $T_>$, and every edge of $\mathcal{P}$ is of type $T_<$.
Since there are at most $r-1$ long intervals and at most $r-1$ short intervals, the path $\mathcal{P}$ contains at least one vertex from a long interval and at least one vertex from a short interval. Since every edge of $\mathcal{P}$ is of type $T_<$, this implies that $u$ is in a short interval and $v$ is in a long interval. This is a contradiction since $uv$ is of type $T_>$.

Now we show that our coloring of $\mathcal{K}_N$ contains no blue copy of $(C_s,\lhd_{mon})$.
Suppose for contradiction that there is such a copy $\mathcal{C}$.
Let $\mathcal{P}$ be the monotone path on $s$ vertices contained in $\mathcal{C}$. Let $u$ be the leftmost vertex of $\mathcal{P}$ and $v$ the rightmost vertex of $\mathcal{P}$.
This time, $\mathcal{C}$ can contain edges between vertices from the same interval. However, $u$ and $v$ belong to different intervals, as no interval contains $s$ vertices.
We distinguish a few cases.

\begin{enumerate}
\item
First, assume that $\mathcal{P}$ contains only edges with both vertices in the same interval, edges of type $T_{<}$, and edges of type $T_{\le}$. Then the vertex indices along $\mathcal{P}$ are nondecreasing from left to right.
By Claim~\ref{claim_there_can_be_only_one}, at most one edge of $\mathcal{P}$ is of type $T_{\le}$.
Thus there is at most one edge of $\mathcal{P}$ between vertices with the same vertex index.
Since every vertex has index at most $s-1$, we see that $\mathcal{P}$ has exactly one edge of type $T_{\le}$ and that the index of $v$ is $s-1$. In particular, $v$ is in a long interval. This implies that from left to right, $\mathcal{P}$ visits a long, a short, and a long interval, in this order.
The distribution of short and long intervals implies that $r$ is odd, $u$ is in a long interval $I_i$, $v$ is in a long interval $I_j$, and $j-i > r-2$. This further implies that $uv$ is of type $T_{>}$, but this contradicts the fact that the index of $u$ is $1$ and the index of $v$ is $s-1$.

\item
In the remaining case, $\mathcal{P}$ has an edge $f$ between intervals $I_i$ and $I_j$ with $j-i>r-2$. There is exactly one such edge since the total number of intervals is $2r-3$.
Every other edge of $\mathcal{P}$ is of type $T_{<}$, or of type $T_{\le}$, or has both vertices in the same interval.
Since $e=uv$ is longer than $f$, it is of type $T_{>}$ or $T_{\ge}$. Therefore the index of $u$ is larger than or equal to the index of $v$. Let $x$ be the left vertex of $f$ and $y$ the right vertex of $f$. Let $\mathcal{P}_1$ be the subpath of $\mathcal{P}$ with endpoints $u$ and $x$, and let $\mathcal{P}_2$ be the subpath of $\mathcal{P}$ with endpoints $y$ and $v$.

\begin{enumerate}
\item
Suppose that $\mathcal{P}$ has no edge of type $T_{\le}$.
Then the indices of vertices in both paths $\mathcal{P}_1$ and $\mathcal{P}_2$ are strictly increasing from left to right. Since $\mathcal{P}_1$ and $\mathcal{P}_2$ have $s-2$ edges in total, it follows that the index of $u$ is equal to the index of $v$, the index of $y$ is $1$ and the index of $x$ is $s-1$. In particular, $x$ is in a long interval and $e$ is of type $T_{\ge}$. This implies that $u$ is in a short interval and $v$ is in a long interval, but this is in contradiction with the distribution of long and short intervals.
In particular, there should be at least $r-2$ intervals between the long intervals that contain $x$ and $v$, but then there is no short interval to the left of the interval containing $x$ if $r$ is odd and there are only at most $r-3$ intervals between two long intervals if $r$ is even.

\item
Suppose that $\mathcal{P}$ has an edge $g$ of type $T_{\le}$. By Claim~\ref{claim_there_can_be_only_one}, there is exactly one such edge. Since $g$ goes from a long interval to a short interval, the edge $e$ cannot go from a short interval to a long interval, by the distribution of long and short intervals. Thus $e$ is of type $T_{>}$.
Consequently, the index of $u$ is larger than the index of $v$. The indices of vertices in both paths $\mathcal{P}_1$ and $\mathcal{P}_2$ are strictly increasing from left to right, with the exception of the edge $g$, whose vertices can have equal indices. It follows that the index of $x$ is $s-1$ and the index of $y$ is $1$. Consequently, $x$ is in a long interval and so $f$ is of type $T_>$. This is a contradiction, since $\mathcal{P}$ cannot have an edge of type $T_>$ together with an edge of type $T_{\le}$, by the distribution of long and short intervals.
In particular, $y$ is in a long interval by Claim~\ref{claim_there_can_be_only_one} and then there is either no short interval to the right nor to the left of $f$ if $r$ is odd or there is no edge of type $T_>$ between two long intervals if $r$ is even.
\end{enumerate}
This finishes the proof that our coloring of $\mathcal{K}_N$ contains no blue copy of $(C_s,\lhd_{mon})$, and the proof of Theorem~\ref{theorem_Ramsey_cycles}. 
\end{enumerate}

Note that we have proved a slightly stronger statement: in our coloring of $\mathcal{K}_N$, there is no red monotone cycle of length at least $r$ and no blue monotone cycle of length at least $s$.

As noted by Cibulka et al.~\cite{cibGao13}, Theorem~\ref{theorem_Ramsey_cycles} implies an exact formula for geometric and convex geometric Ramsey numbers of cycles (see Subsection~\ref{subsection_motivation} for definitions). 

\begin{corollary}
\label{cor_Karolyi}
For every integer $n \ge 3$, we have $\Rc(C_n)=\Rg(C_n)=2n^2 - 6n + 6$.
\end{corollary}
\begin{proof}
We recall that $\Rc(G) \le \Rg(G)$ for every outerplanar graph $G$.
The upper bound $\Rg(C_n) \le 2n^2-6n+6$ was proved by K\'arolyi et al.~\cite{kar98}. The lower bound $2n^2-6n+6 \le \Rc(C_n)$ follows from Observation~\ref{obsDiscreteGeo} and Theorem~\ref{theorem_Ramsey_cycles}.
\end{proof}


\subsection{Lower bound for matchings}
Here we prove Theorem~\ref{veta_parovani}.

Let $r\ge 3$  and let $R_r\mathrel{\mathop:}=\R(K_r)-1$. 
We construct a sequence of ordered matchings $\mathcal{M}_{r,k}$, $k \ge 1$, with $n_{r,k}$ vertices and a sequence of $2$-colorings $c_{r,k}$ of ordered complete graphs $\mathcal{K}_{N_{r,k}}$ such that $c_{r,k}$ avoids $\mathcal{M}_{r,k}$. Then we choose $k(r)$ so that $n_{r,k(r)}$ is exponential in $r$. This will imply that $N_{r,k(r)}$ is superpolynomial in $n_{r,k(r)}$ when $r$ grows to infinity. 

First we show an inductive construction of the colorings $c_{r,k}$.
Let $N_{r,1} \mathrel{\mathop:}= R_r$ and let $c_{r,1}$ be a $2$-coloring of $\mathcal{K}_{N_{r,1}}$ avoiding $\mathcal{K}_r$.
Let $k\ge 1$ and suppose that a coloring $c_{r,k}$ of $\mathcal{K}_{N_{r,k}}$ has been constructed.
Let $N_{r,k+1}\mathrel{\mathop:}= R_r \cdot N_{r,k}$.
Partition the vertex set of $\mathcal{K}_{N_{r,k+1}}$ into $R_r$ disjoint consecutive intervals $I_1,\allowbreak I_2,\allowbreak \dots,\allowbreak I_{R_r}$, each of size $N_{r,k}$.
Color the complete subgraph induced by each $I_i$ by $c_{r,k}$.
The remaining edges of $\mathcal{K}_{N_{r,k+1}}$ form a complete $R_r$-partite ordered graph $\mathcal{F}_{r,k+1}$, which can be colored to avoid $\mathcal{K}_r$ in the following way.
Suppose that $v_1, v_2, \dots, v_{R_r}$ are the vertices of $\mathcal{K}_{N_{r,1}}$.
Then for every $i,j$, $1\le i<j\le R_r$, and for every edge $e$ of $\mathcal{F}_{r,k+1}$ with one vertex in $I_i$ and the other vertex in $I_j$, let $c_{r,k+1}(e)\mathrel{\mathop:}= c_{r,1}(v_iv_j)$.
Clearly, $N_{r,k}=(R_r)^k$ for every $k\ge 1$.

The matchings $\mathcal{M}_{r,k}$ are also constructed inductively.
We start with constructing $\mathcal{M}_{r,1}$, which serves as a basic building block. Roughly speaking, we expand the vertices of $\mathcal{K}_r$ to form a matching and take $R_r$ shifted copies of this matching; see Figure~\ref{fig_matching_basic}.
More precisely, consider the integers $1,2,\dots, r^2R_r$ as vertices, and let $l_i \mathrel{\mathop:}= (i-1)rR_r$, for $1\le i\le r$.
For every pair $i,j$, where $1\le i<j\le r$, we add the $R_r$ edges $\{l_i+j,l_j+i\},\allowbreak \{l_i+j+r,l_j+i+r\},\allowbreak \{l_i+j+2r,l_j+i+2r\},\allowbreak  \dots,\allowbreak \{l_i+j+(R_r-1)r,l_j+i+(R_r-1)r\}$.
Note that the vertices $l_i + i + mr$, where $1\le i \le r$ and $0 \le m < R_r$, are isolated. After removing these vertices we obtain an ordered matching $\mathcal{M}_{r,1}$ with $t_r \mathrel{\mathop:}= r(r-1)R_r$ vertices. 

\begin{figure}
\centering
\includegraphics{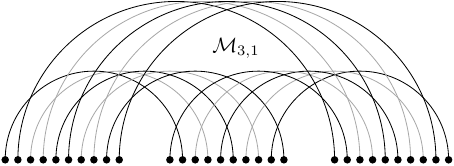} 
\caption{The matching $\mathcal{M}_{3,1}$.}
\label{fig_matching_basic}
\end{figure}

Let $n_{r,1} \mathrel{\mathop:}= t_r$.
Now let $k\ge 1$ and suppose that $\mathcal{M}_{r,k}$ has been constructed.
Let $J_1,\allowbreak L_1,\allowbreak J_2,\allowbreak L_2,\allowbreak \dots,\allowbreak L_{r-1},\allowbreak J_r$ be an ordered sequence of disjoint intervals of vertices, of size $|L_i|=n_{r,k}$ and $|J_i|=(r-1)R_r$.
The matching $\mathcal{M}_{r,k+1}$ is obtained by placing a copy of $\mathcal{M}_{r,k}$ on each of the $r-1$ intervals $L_i$ and a copy of $\mathcal{M}_{r,1}$ on the union of the $r$ intervals $J_i$.
See Figure~\ref{fig_matching_step}.
We have $n_{r,k+1}=(r-1)n_{r,k} + t_r$. 

\begin{figure}
\centering
\includegraphics{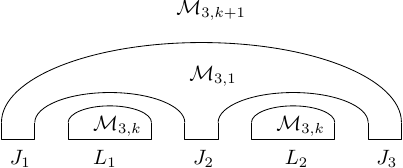} 
\caption{The construction of $\mathcal{M}_{3,k+1}$.}
\label{fig_matching_step}
\end{figure}

Now we show that for every $k$, the coloring $c_{r,k}$ of $\mathcal{K}_{N_{r,k}}$ avoids $\mathcal{M}_{r,k}$.
Trivially, $c_{r,1}$ avoids $\mathcal{M}_{r,1}$ since $n_{r,1}=t_r>R_r=N_{r,1}$.
Let $k\ge 1$ and suppose that $c_{r,k}$ avoids $\mathcal{M}_{r,k}$.
Let $I_1, \dots, I_{R_r}$ be the intervals of vertices of $\mathcal{K}_{N_{r,k+1}}$ from the construction of $c_{r,k+1}$ and let $J_1,\allowbreak L_1,\allowbreak \dots,\allowbreak L_{r-1},\allowbreak J_r$ be the intervals of vertices of $\mathcal{M}_{r,k+1}$ from the construction of $\mathcal{M}_{r,k+1}$.
Let the edges of $\mathcal{K}_{N_{r,k+1}}$ be colored by $c_{r,k+1}$.
Consider a copy of $\mathcal{M}_{r,k+1}$ in $\mathcal{K}_{N_{r,k+1}}$.
If two intervals $J_j$ and $J_{j+1}$ intersect some interval $I_i$, then $L_j \subset I_i$.
Since $L_j$ induces $\mathcal{M}_{r,k}$ in $\mathcal{M}_{r,k+1}$ and $I_i$ induces $\mathcal{K}_{N_{r,k}}$ colored with $c_{r,k}$ in $\mathcal{K}_{N_{r,k+1}}$, the copy of $\mathcal{M}_{r,k+1}$ is not monochromatic by induction.
Thus we may assume that every interval $I_i$ is intersected by at most one interval $J_j$.

Partition each interval $J_j$ into $R_r$ intervals $J_j^1,\allowbreak J_j^2,\allowbreak \dots,\allowbreak J_j^{R_r}$ of length $r-1$, in this order.
At most $R_r-1$ of the $rR_r$ intervals $J_j^l$ contain vertices from at least two intervals $I_i$, $1\le i\le R_r$.
Thus there is an $l$ such that for every $j$, $1\le j\le r$, the whole interval $J_j^l$ is contained in some interval $I_{i(j)}$.
Moreover, all the intervals $I_{i(j)}$ are pairwise distinct by our assumption.
By the construction of $\mathcal{M}_{r,k+1}$, there is exactly one edge $e_{j,j'}$ in $\mathcal{M}_{r,k+1}$ between every pair of intervals $J_j^l$, $J_{j'}^l$. 
By the coloring of $\mathcal{F}_{r,k+1}$, we have $c_{r,k+1}(e_{j,j'}) = c_{r,1}(v_{i(j)}v_{i(j')})$.
Since the edges $v_{i(j)}v_{i(j')}$ form a complete subgraph with $r$ vertices in $\mathcal{K}_{N_{r,1}}$ and $c_{r,1}$ avoids $\mathcal{K}_r$, the copy of $\mathcal{M}_{r,k+1}$ in $\mathcal{K}_{N_{r,k+1}}$ is not monochromatic.
Thus $c_{r,k+1}$ avoids $\mathcal{M}_{r,k+1}$.

Solving the recurrence for $n_{r,k}$, we get 
\[
n_{r,k}= (1+(r-1)+\dots +(r-1)^{k-1})\cdot t_r < (r-1)^k \cdot t_r < r^{k+2} \cdot R_r.
\]

Now we assume that $r$ is sufficiently large and we choose $k(r)$ as follows.
Let $c \mathrel{\mathop:}= (\log{R_r}) / r$, where we recall that $\log$ denotes the base 2 logarithm.
By \eqref{eq_klasicke_odhady_ramsey}, we have $c\in [1/2, 2)$.
Let $k(r)\mathrel{\mathop:}=\lfloor (cr/\log{r})-2\rfloor=(cr/\log{r})-2-\varepsilon$, where $\varepsilon \in [0,1)$.
Let $n \mathrel{\mathop:}= n_{r,k(r)}$, $N \mathrel{\mathop:}= N_{r,k(r)}$ and $\mathcal{M}  \mathrel{\mathop:}=\mathcal{M}_{r,k(r)}$. 
We have
\begin{align*}
n&=n_{r,k(r)}<r^{k(r)+2}\cdot R_r\le 2^{cr+\log R_r}=2^{2cr} \hskip 0.3cm \text{ and } \\ 
N&=N_{r,k(r)}=(R_r)^{k(r)}=2^{crk(r)}>2^{(c^2r^2/\log{r})-3cr}.
\end{align*}
Using these bounds together with the trivial bound $2^{cr}=R_r<n$, we get
\begin{align*}
\log{N}-\frac{\log^2{n}}{5\log\log{n}} &> \frac{c^2r^2}{\log{r}}-3cr - \frac{4c^2r^2}{5(\log{r}+\log{c})} \\
&= c^2r^2 \left( \frac{1}{\log{r}} - \frac{3}{cr} - \frac{4}{5(\log{r}+\log{c})}  \right)  \\
&>0
\end{align*}
where the last inequality is satisfied for $r>540$.
The theorem follows.

\medskip
We remark that our colorings $c_{r,k}$ of $\mathcal{K}_{N_{r,k}}$ are not constructive, since we use the probabilistic lower bound from Ramsey's theorem.

\section{Upper Bounds}\label{section_upper_bounds}

\subsection{Proof of Theorem~\ref{veta_rozlozitelne}}
\label{section_proofRozlozitelne}

We prove the following general form of Theorem~\ref{veta_rozlozitelne}, which allows us to use induction.

\begin{theorem}
\label{veta_rozlozitelne_dvojrozmerna}
For fixed positive integers $k$, $q \ge 2$ and $(k,q)$-decomposable ordered graphs $\mathcal{G}$ and $\mathcal{H}$ with $r$ and $s$ vertices, respectively, we have 
\[
\Ro(\mathcal{G},\mathcal{H})\le C_k \cdot 2^{64k(\lceil \log_{q/(q-1)}{r} \rceil+\lceil \log_{q/(q-1)}{s} \rceil)}
\] 
where $C_k$ is a sufficiently large constant with respect to $k$.
\end{theorem}

\begin{lemma}
\label{lemma_neighbors}
Let the edges of $\mathcal{K}_N$ be colored red and blue.
Then there is a set $U$ with at least $\lfloor N/(16\cdot10^5) \rfloor$ vertices of $\mathcal{K}_N$ satisfying at least one of the following conditions:
\begin{enumerate}[label=(\alph*)]
\item\label{item-lemNeigh1} every vertex of $U$ has at least $N/11$ blue neighbors to the left and $N/11$ blue neighbors to the right of $U$,
\item\label{item-lemNeigh2} every vertex of $U$ has at least $N/11$ red neighbors to the left and $N/11$ red neighbors to the right of $U$.
\end{enumerate}
\end{lemma}

\begin{proof}
We assume that $N \ge 16\cdot10^5$, otherwise the statement is trivial. We define the following two conditions for a vertex $v$ of $\mathcal{K}_N$:
\begin{enumerate}[label=(\roman*)]
\item\label{item-neigh1} $v$ has at least $\frac{20}{217}N$ blue left and at least $\frac{20}{217}N$ blue right neighbors,
\item\label{item-neigh2} $v$ has at least $\frac{20}{217}N$ red left and at least $\frac{20}{217}N$ red right neighbors.
\end{enumerate}
First, we show that there is a set $W$ with at least $N/2000$ vertices such that either every vertex of $W$ satisfies~\ref{item-neigh1} or every vertex of $W$ satisfies~\ref{item-neigh2}.
Let $B$ be the set of vertices of $\mathcal{K}_N$ that satisfy the condition~\ref{item-neigh1} and let $R$ be the set of vertices of $\mathcal{K}_N$ that satisfy \ref{item-neigh2}.
Suppose that $|B|<N/2000$ and $|R|<N/2000$, otherwise we are done.

Let $\mathcal{K'}$ be the ordered complete graph obtained from $\mathcal{K}_N$ by removing the vertices of $B \cup R$.
From the assumptions $\mathcal{K'}$ has more than $(1-\frac{2}{2000})N=\frac{999}{1000}N$ vertices and contains no monochromatic ordered star $\mathcal{S}_{t,t}$ for $t  \mathrel{\mathop:}=  \left\lceil\frac{20}{217}N\right\rceil+1$.
Therefore $\mathcal{K'}$ has fewer than $\Ro(\mathcal{S}_{t,t},\mathcal{S}_{t,t})$ vertices.

Using Theorem~\ref{thmStarsPair} and the fact that $\Ro(\mathcal{S}_{t,1},\mathcal{S}_{t,t}) = \Ro(\mathcal{S}_{1,t},\mathcal{S}_{t,t})$, we have
\begin{align*}
\Ro(\mathcal{S}_{t,t},\mathcal{S}_{t,t}) & = \Ro(\mathcal{S}_{t,1},\mathcal{S}_{t,t}) + \Ro(\mathcal{S}_{1,t},\mathcal{S}_{t,t})-1 = 2(\Ro(\mathcal{S}_{1,t},\mathcal{S}_{t,1})+2t-3)-1 \\
& = 2\left(\left\lfloor\frac{-1+\sqrt{1+8(t-2)^2}}{2} \right\rfloor + 4t -5\right)-1 < (8 + 2\sqrt{2})t.
\end{align*}

Altogether we have 
$|V(\mathcal{K'})|< (8+2\sqrt{2})(\left\lceil\frac{20}{217}N\right\rceil+1) < \frac{999}{1000}N < |V(\mathcal{K'})|$, a contradiction.
Thus there is a set $W$ such that all its vertices satisfy one of the two conditions, say,~\ref{item-neigh1}.

Now, we find the set $U$ as a subset of $W$.
To do so, we partition the vertex set of $\mathcal{K}_N$ into $\frac{16\cdot10^5}{2000} = 800$ intervals $I_1,\dots,I_{800}$ such that each contains at least $\lfloor N/(16\cdot10^5)\rfloor$ vertices of $W$.
This is possible as $|W|\ge N/2000$.
Clearly, there is an interval $I_i$ with at most $N/800$ vertices of $\mathcal{K}_N$.
We set $U  \mathrel{\mathop:}= I_i \cap W$.

Since every vertex of $U$ has at least $\frac{20}{217}N$ blue left neighbors, it also has at least $\frac{20}{217}N-N/800 > N/11$ blue neighbors to the left of $I_i$ and thus to the left of $U$. Analogously, every vertex of $U$ has at least $N/11$ blue neighbors to the right of $U$. Therefore, $U$ satisfies condition~\ref{item-lemNeigh1} of the lemma.
\end{proof}

We use the following two classical results further in the proof.
The K\H{o}v\'{a}ri--S\'{o}s--Tur\'{a}n theorem~\cite{kovari54} gives an upper bound on the maximum number of edges in a bipartite graph that contains no copy of a given complete bipartite graph.

\begin{theorem}[\cite{boll04, hylten58, kovari54}]
\label{thmKovariTuranSos}
Let $Z(m,n;s,t)$ be the maximum number of edges in a bipartite graph $G=(A \cup B, E)$ with $|A|=m$ and $|B|=n$ that does not contain $K_{s,t}$ as a subgraph with $s$ vertices in $A$ and $t$ vertices in $B$. Assuming $2\le s \le m$ and $2 \le t \le n$,  we have 
\[
Z(m,n;s,t) < (s-1)^{1/t}(n-t+1)m^{1-1/t}+(t-1)m < s^{1/t}nm^{1-1/t}+tm.
\]
\end{theorem}

Erd\H{o}s and Szekeres proved the following upper bound on off-diagonal Ramsey numbers of complete graphs.

\begin{theorem}[\cite{ErSz35_a_comb_problem}]
\label{theorem_rams_off_diagonal}
For every $r,s \ge 2$, we have $\R(K_r,K_s) \le \binom{r+s-2}{r-1} \le (r+s)^r \le (rs)^r.$
\end{theorem}

By Observation~\ref{obs_uplnaky}, we have the same upper bound for the ordered Ramsey numbers $\Ro(\mathcal{K}_r,\mathcal{K}_s)$.

\paragraph{Proof of Theorem~\ref{veta_rozlozitelne_dvojrozmerna}.}
Let $\mathcal{G}$ and $\mathcal{H}$ be $(k,q)$-decomposable ordered graphs with $r$ and $s$ vertices, respectively. 
Let $N =  N_{k,q}(r,s)\mathrel{\mathop:}= C_k \cdot 2^{64k(\lceil \log_{q/(q-1)} r \rceil+\lceil \log_{q/(q-1)} s \rceil)}$ where $C_k$ is a constant sufficiently large with respect to $k$. Assume that the edges of $\mathcal{K}_N$ are colored red and blue.
We show that there is a blue copy of $\mathcal{G}$ or a red copy of $\mathcal{H}$ in $\mathcal{K}_N$. We proceed by induction on $\lceil \log_{q/(q-1)}{r}\rceil + \lceil \log_{q/(q-1)}{s} \rceil$.
For the induction basis, we assume that either $\lceil\log_{q/(q-1)}{r}\rceil=0$ or $\lceil\log_{q/(q-1)}{s}\rceil=0$.
In these cases we have $r =1$ or $s =1$, respectively, and the statement is trivial.

Now assume that the theorem is true for every pair $\mathcal{G}'$, $\mathcal{H}'$ of $(k,q)$-decompo\-sable ordered graphs with $r'$ and $s'$ vertices, respectively, such that $\lceil\log_{q/(q-1)}{r'}\rceil + \lceil\log_{q/(q-1)}{s'}\rceil <\lceil\log_{q/(q-1)}{r}\rceil + \lceil\log_{q/(q-1)}{s}\rceil$.

Let $U$ be the subset of vertices of $\mathcal{K}_N$ from Lemma~\ref{lemma_neighbors}.
Without loss of generality, we assume that $U$ satisfies part~\ref{item-lemNeigh1} of the lemma.
That is, $U$ has at least $\lfloor N/(16\cdot10^5) \rfloor$ vertices such that each of them has at least $N/11$ blue neighbors to the left and $N/11$ blue neighbors to the right of $U$.

By Theorem~\ref{theorem_rams_off_diagonal}, there is a blue copy of $\mathcal{K}_{61k}$ or a red copy of $\mathcal{K}_s$ in $\mathcal{K}_N[U]$ if $\lvert U\rvert\ge (61ks)^{61k}$.
This condition is satisfied if $C_k\ge 16\cdot 10^5\cdot(61k)^{61k}$, since 
$(61ks)^{61k}  \le (61k)^{61k}\cdot 2^{61k\cdot\log s} \le (61k)^{61k}\cdot 2^{61k(\log_{q/(q-1)} s)}$.
If $\mathcal{K}_N[U]$ contains a red copy of $\mathcal{K}_s$, we are done.
Thus, assume that $\mathcal{K}_N[U]$ contains a blue copy of $\mathcal{K}_{61k}$, and let $U_1 \subset U$ be its vertex set.

Next we will apply Theorem~\ref{thmKovariTuranSos} to obtain a set $U_2 \subset U_1$ of size $6k$ whose vertices have at least $N/2^{64k}$ common blue neighbors to the left of $U$.
Then we apply Theorem~\ref{thmKovariTuranSos} again to obtain a set $V \subset U_2$ of size $k$ whose vertices have at least $N/2^{64k}$ common blue neighbors to the right of $U$.

Let $J_L$ be the interval of vertices of $\mathcal{K}_N$ that are to the left of $U$ and $J_R$ the interval of vertices of $\mathcal{K}_N$ that are to the right of $U$.
By the construction of $U$, we have the trivial bound $|J_L|,|J_R| \ge N/11$, and thus $|J_L|,|J_R|\le 10N/11$.
Without loss of generality, we assume that $|J_R| \le N/2$.

Since $|J_L| \le 10N/11$, the number of blue edges between $J_L$ and $U_1$ is at least $(N/11)\cdot |U_1| \ge |J_L|\cdot|U_1|/10$.
By Theorem~\ref{thmKovariTuranSos}, we have 
\begin{align*}
Z(|J_L|,|U_1|;|J_L|/2^{60k},6k) &< (|J_L|/2^{60k})^{1/(6k)}\cdot 61k \cdot|J_L|^{1-1/(6k)}+6k\cdot|J_L| \\
&= |J_L|\cdot (61k/2^{10} + 6k) \le |J_L|\cdot 61k/10 = \frac{|J_L|\cdot|U_1|}{10}.
\end{align*}
Thus, there is a blue complete bipartite graph between at least $|J_L|/2^{60k}$ vertices in $J_L$ and $6k$ vertices in $U_1$. These $6k$ vertices form the set $U_2$.

Since $|J_R| \le N/2$, the number of blue edges between $U_2$ and $J_R$ is at least $(N/11) \cdot |U_2|\ge |U_2|\cdot |J_R|\cdot 2 / 11$.
By Theorem~\ref{thmKovariTuranSos}, we have 
\begin{align*}
Z(|J_R|,|U_2|;|J_R|/2^{7k},k) &< (|J_R|/2^{7k})^{1/k}\cdot 6k \cdot|J_R|^{1-1/k}+k\cdot|J_R| \\
&= |J_R|\cdot (6k/2^{7} + k) \le |J_R|\cdot 6k \cdot 2/11 = \frac{2|J_R|\cdot|U_2|}{11}.
\end{align*}
Thus, there is a blue complete bipartite graph between at least $|J_R|/2^{7k}$ vertices in $J_R$ and $k$ vertices in $U_2$. These $k$ vertices form the set $V$. Since $|J_L|,|J_R| \ge N/11$, the vertices of $V$ have at least $N/(2^{60k}\cdot11)>N/2^{64k}$ common blue neighbors to the left of $V$ and at least $N/(2^{7k}\cdot11)>N/2^{64k}$ common blue neighbors to the right of $V$.

Since $\mathcal{G}$ is $(k,q)$-decomposable, we can partition the vertices of $\mathcal{G}$ into three intervals $I_L$, $I$, and $I_R$ where $0<|I| \le k$ and $|I_L|,|I_R| \le r(q-1) / q$ such that $I$ is to the right of $I_L$ and to the left of $I_R$,
the intervals $I_L$ and $I_R$ induce $(k,q)$-decomposable ordered graphs $\mathcal{G}_L$ and $\mathcal{G}_R$, respectively, and there is no edge between $\mathcal{G}_L$ and $\mathcal{G}_R$. 

From our choice of $N$, we have 
\begin{align*}
N/2^{64k} & =C_k \cdot 2^{64k(\lceil \log_{q/(q-1)}{r} \rceil+\lceil \log_{q/(q-1)}{s} \rceil - 1)}\\
&=C_k \cdot 2^{64k(\lceil \log_{q/(q-1)}{r(q-1)/q} \rceil+\lceil \log_{q/(q-1)}{s} \rceil)}\ge N_{k,q}(\lfloor r(q-1)/q\rfloor,s)
\end{align*}
and so $\Ro(\mathcal{G}_L,\mathcal{H}), \Ro(\mathcal{G}_R,\mathcal{H}) \le N/2^{64k}$.
Therefore, using the inductive assumption, we can find either a blue copy of $\mathcal{G}_L$ or a red copy of $\mathcal{H}$ in the common blue left neighborhood of~$V$.
Similarly, we can find a blue copy of $\mathcal{G}_R$ or a red copy of $\mathcal{H}$ in the common blue right neighborhood of $V$.
Suppose that we do not obtain a red copy of $\mathcal{H}$ in any of these two cases.
Then we find a blue copy of $\mathcal{G}$ by choosing $|I|$ vertices of $V$ as a copy of $I$ and connect them to the blue copies of $\mathcal{G}_L$ and $\mathcal{G}_R$. \qed


\subsection{Proof of Theorem~\ref{thmIntChromDiag}}
\label{section_proofIntChrom}

We derive Theorem~\ref{thmIntChromDiag} as a consequence of a stronger Theorem~\ref{thmIntChrom}, which gives an upper bound on off-diagonal ordered Ramsey numbers of graphs with constant interval chromatic number.

The first step in the proof of  Theorem~\ref{thmIntChrom} is the following lemma, whose proof is motivated by the proof of the upper bound on ordered Ramsey numbers of ordered paths by Cibulka et al.~\cite{cibGao13}.

\begin{lemma}
\label{lemmaIntChrom2}
Let $k,t,n$ be positive integers and let $\mathcal{G}$ be an ordered $k$-degenerate graph on $n$ vertices.
Then $\Ro(\mathcal{G},\mathcal{K}_{t,t}) \le n^2 t^{k+1}$.
\end{lemma}

\begin{proof}
Assume that $\mathcal{G}=(G,\prec)$. Let $N\mathrel{\mathop:}= n^2 t^{k+1}$ and assume that the edges of $\mathcal{K}_N$ are colored red and blue.
We partition the vertices of $\mathcal{K}_N$ into $n$ disjoint consecutive intervals of length $n t^{k+1}$.
The $i$th such interval is denoted by $I(v)$ where $v$ is the $i$th vertex of $\mathcal{G}$ in the ordering $\prec$.

We try to construct a blue copy $h(\mathcal{G})$ of $\mathcal{G}$ in $\mathcal{K}_N$ in $n$ steps.
In each step of the construction we find an image $h(w) \in I(w)$ of a new vertex $w$ of $\mathcal{G}$ or a red copy of $\mathcal{K}_{t,t}$.

For every vertex $v$ of $\mathcal{G}$ that has no image $h(v)$ yet, we keep a set $U(v) \subseteq I(v)$ of possible candidates for $h(v)$.
At the beginning we set $U(v)\mathrel{\mathop:}= I(v)$ for every $v \in V(G)$.
Throughout the proof, we will keep the property that the size of $U(v)$ is a multiple of $t$. 

Let $\lessdot$ be an ordering of the vertices of $\mathcal{G}$ such that every vertex $v$ of $\mathcal{G}$ has at most $k$ left neighbors in $\lessdot$.
This ordering exists as $\mathcal{G}$ is $k$-degenerate.
Note that the ordering $\lessdot$ might differ from the ordering $\prec$.

Let $w$ be the leftmost vertex of $\mathcal{G}$ in the ordering $\lessdot$ that has no image $h(w)$ yet.
Suppose that $u_1,\dots,u_s \in V(G)$ are the right neighbors of $w$ in $\lessdot$.
We show how to find the image $h(w)$ or a red copy of $\mathcal{K}_{t,t}$ in $\mathcal{K}_N$.

Let $i\in [s]$.
We claim that in $U(w)$ every vertex except for at most $t-1$ vertices has at least $|U(u_i)|/t$ blue neighbors in $U(u_i)$ or there is a red copy of $\mathcal{K}_{t,t}$ with edges between $U(w)$ and $U(u_i)$.

Suppose first that there is a subset $W\subseteq U(w)$ of size $t$ such that each vertex of $W$ has fewer than $|U(u_i)| / t$ blue neighbors in $U(u_i)$.
In such a case we delete from $U(u_i)$ every vertex that is a blue neighbor of some vertex of $W$.
Afterwards, there are still at least 
\[|U(u_i)|-|W|\cdot \left(\frac{|U(u_i)|}{t}-1\right) = |U(u_i)| - t\cdot \left(\frac{|U(u_i)|}{t}-1\right) = t\]
vertices left in $U(u_i)$ and every such vertex has only red neighbors in $W$.
Thus we have a red copy of $\mathcal{K}_{t,t}$ in $\mathcal{K}_N$.

By our claim, there is a red copy of $\mathcal{K}_{t,t}$ in $\mathcal{K}_N$ or a set $Z(w) \subseteq U(w)$ of size at least $|U(w)|-s(t-1)>|U(w)|-nt$ such that for every $i\in [s]$, every vertex of $Z(w)$ has at least $|U(u_i)| / t$ blue neighbors in $U(u_i)$.
We may assume that the latter case occurs, as otherwise we are done.

We choose an arbitrary vertex $h(w)$ of $Z(w)$ to be the image of $w$ in the constructed blue copy $h(\mathcal{G})$ of $\mathcal{G}$.
For this we need to know that $Z(w)$ is nonempty; we show this at the end of the proof.
For every $i\in [s]$, we update the set $U(u_i)$ to be a set of $|U(u_i)|/t$ blue neighbors of $h(w)$ in $U(u_i)$. 

After these updates, we choose the first vertex in $\lessdot$ that does not have an image yet and proceed with the next step. If every vertex of $\mathcal{G}$ has an image, then we have found a blue copy of $\mathcal{G}$.

It remains to show that the set $Z(w)$ is nonempty in each step.
Since $w$ has at most $k$ left neighbors in $\lessdot$, we have updated $U(w)$ at most $k$ times.
The size of $U(w)$ is initially $nt^{k+1}$ and it is divided by $t$ in every update.
Thus, in the end, $|U(w)|\ge nt$. Consequently, $|Z(w)|>|U(w)|-nt\ge 0$.
\end{proof}

Let $\mathcal{K}_p(n)$ be the ordered complete $p$-partite graph with parts of size $n$ forming consecutive intervals.

\begin{theorem}
\label{thmIntChrom}
Let $k$, $n$, and $p \geq 2$ be positive integers and let $\mathcal{G}$ be an ordered $k$-degenerate graph on $n$ vertices.
Then 
\[\Ro(\mathcal{G},\mathcal{K}_p(n)) \le n^{(1+2/k)(k+1)^{\lceil \log{p}\rceil}-2/k}.\]
\end{theorem}

\begin{proof}
First, we define a function $f_{k,n}(q) \colon \mathbb{N} \to \mathbb{N}$ as
\[
f_{k,n}(q)\mathrel{\mathop:}= n^{(1+2/k)(k+1)^q-2/k}.
\]
This function satisfies the recurrence $f_{k,n}(1)=n^{k+3}$ and $f_{k,n}(q)=n^2\cdot (f_{k,n}(q-1))^{k+1}$ for every integer $q \ge 2$.

We assume without loss of generality that $p=2^q$ for some positive integer $q$.
We proceed by induction on $q$.
The case $q=1$ follows immediately from Lemma~\ref{lemmaIntChrom2} applied with $t\mathrel{\mathop:}=n$.

Now let $q\ge 2$.
Let $\mathcal{K}_N$ be an ordered complete graph with $N\mathrel{\mathop:}= f_{k,n}(q)$ vertices and edges colored red and blue.
We show that there is always a blue copy of $\mathcal{G}$ or a red copy of $\mathcal{K}_p(n)$ in $\mathcal{K}_N$.

According to Lemma~\ref{lemmaIntChrom2}, there is a blue copy of $\mathcal{G}$ or a red copy of $\mathcal{K}_{t,t}$ for $t\mathrel{\mathop:}= f_{k,n}(q-1)$.
In the first case we are done, thus we assume that the latter case occurs.
Let $A$ be the left part of size $t$ and $B$ the right part of size $t$ in the red copy of $\mathcal{K}_{t,t}$.

Since the induced ordered subgraph $\mathcal{K}_N[A]$ has $f_{k,n}(q-1)$ vertices, there is a blue copy of $\mathcal{G}$ or a red copy of $\mathcal{K}_{p/2}(n)$ in $\mathcal{K}_N[A]$ by the inductive assumption.
An analogous statement holds for the ordered subgraph $\mathcal{K}_N[B]$. 

Thus, if there is no blue copy of $\mathcal{G}$ in $\mathcal{K}_N[A]$ and in $\mathcal{K}_N[B]$, then the two red copies of $\mathcal{K}_{p/2}(n)$ together with the red edges between $\mathcal{K}_N[A]$ and $\mathcal{K}_N[B]$ form a red copy of $\mathcal{K}_p(n)$ in~$\mathcal{K}_N$.
\end{proof}

\section{Fixed ordered graph, variable number of colors}
\label{section_variable_c}

Here we discuss the asymptotics of ordered Ramsey numbers $\Ro(\mathcal{G}; c)$ of a fixed ordered graph $\mathcal{G}$ as a function of the number of colors.
That is, for the rest of the section we assume that $\mathcal{G}$ is a fixed ordered graph and that the number $c$ of colors can be arbitrarily large.

The unordered Ramsey numbers are at most polynomial for bipartite graphs and at least exponential other\-wise; this follows from the K\H{o}v\'{a}ri--S\'{o}s--Tur\'{a}n theorem (Theorem~\ref{thmKovariTuranSos}) and from the exis\-tence of a decomposition of $K_n$ into $\lceil\log{n}\rceil$ bipartite subgraphs, respectively.
For ordered Ramsey numbers we observe a similar dichotomy, but the characterization is more subtle.

An ordered graph $\mathcal{G}$ is \emph{separable} if the vertex set of $\mathcal{G}$ can be partitioned into two nonempty intervals $I_1,I_2$ such that there is no edge between $I_1$ and $I_2$. An ordered graph is \emph{nonseparable} if it is not separable.

We find that $\Ro(\mathcal{G}; c)$ is exponential in $c$ if $\mathcal{G}$ contains a nonseparable ordered graph with interval chromatic number $3$, and polynomial otherwise. Moreover, there are only finitely many minimal nonseparable ordered graphs with interval chromatic number at least $3$. Therefore, the class of ordered graphs with polynomial ordered Ramsey numbers can be characterized by a finite number of forbidden ordered subgraphs.

\begin{figure}
\centering
 \includegraphics{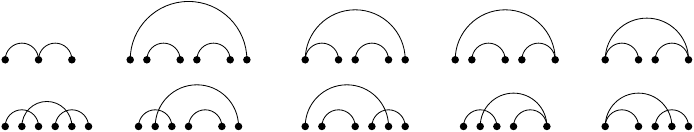} 
\caption{Minimal nonseparable ordered graphs with interval chromatic number at least $3$.}
\label{fig_dichotomy}
\end{figure}

\paragraph{Proof of Theorem~\ref{thm-dichotomy}.}
For part~\ref{item-dichotomy1}, let $\mathcal{G}$ be a given $n$-vertex ordered graph contained in $n\cdot \mathcal{K}_{n,n}$.
For $N \mathrel{\mathop:}= (2cn)^{n+1}$, let the edges of $\mathcal{K}_N$ be colored with $c$ colors.
We find a monochromatic copy of $\mathcal{G}$.

For $t \mathrel{\mathop:}= cn$, we partition the vertex set of $\mathcal{K}_N$ into $2t$ intervals $A_1,B_1,\dots,A_t,\allowbreak B_t$ in this order, such that each interval has size $K \mathrel{\mathop:}=(2cn)^n$.
For every $i=1,\ldots,t$, it follows from the pigeonhole principle that there is a color $c_i$ that colors at least $K^2/c$ edges of $\mathcal{\mathcal{K}_N}[A_i \cup B_i]$.

By Theorem~\ref{thmKovariTuranSos}, we have $Z(K,K;n,n) < 2nK^{2-1/n}= K^2/c$.
Consequently, for every $i=1,\ldots,t$, there is a copy of $\mathcal{K}_{n,n}$ of color $c_i$ in $\mathcal{\mathcal{K}_N}[A_i \cup B_i]$.
By the pigeonhole principle, we have a monochromatic copy of $n \cdot \mathcal{K}_{n,n}$.
Since $\mathcal{G} \subseteq n\cdot \mathcal{K}_{n,n}$, we have a monochromatic copy of $\mathcal{G}$ as well.

To prove part~\ref{item-dichotomy2}, we first show that if $\mathcal{G}$ is not contained in $n \cdot \mathcal{K}_{n,n}$, then $\mathcal{G}$ contains one of the ordered graphs from Figure~\ref{fig_dichotomy}.

The ordered graph $\mathcal{G}$ contains a nonseparable ordered graph $\mathcal{H}$ with interval chromatic number $t \ge 3$, since $\mathcal{G}$ is not an ordered subgraph of $n \cdot \mathcal{K}_{n,n}$.
Let $I_1,\ldots,I_t$ be a partitioning of the vertex set of $\mathcal{H}$ into $t$ consecutive  intervals such that there is no edge of $\mathcal{H}$ with both vertices in the same interval.
Then $\mathcal{H}$ has an edge $e$ between intervals $I_1$ and $I_2$ and an edge $f$ between intervals $I_2$ and $I_3$.
If $e$ and $f$ share a vertex, they form a monotone path on three vertices, which is the first ordered graph in Figure~\ref{fig_dichotomy}.

Assume that no vertex of $I_2$ has a neighbor in both $I_1$ and $I_3 \cup \cdots \cup I_t$.
Then we partition $I_2$ into sets $A_1$, $A_2$, and $A_3$ such that every vertex of $A_1$ has a neighbor in $I_1$, no vertex in $A_2$ has a neighbor in $I_1 \cup I_3$, and every vertex of $A_3$ has a neighbor in $I_3$.
If $A_3$ is to the left of $A_1$, then we can move some vertices of $I_2$ into $I_1$ and some into $I_3$ to obtain a partitioning of the vertex set of $\mathcal{H}$ into $t-1$ intervals such that there is no edge with both vertices in the same interval.
This is impossible, as the interval chromatic number of $\mathcal{H}$ is~$t$.
Thus we can assume that the vertex in $e \cap I_2$ is to the left of the vertex in $f \cap I_2$ and that every vertex between $e \cap I_2$ and $f \cap I_2$ lies in $A_2$.

Since $\mathcal{H}$ is nonseparable, there is an edge $g$ of $\mathcal{H}$ with one vertex to the left of $e \cap I_2$ and the other one to the right of $f \cap I_2$.
The left vertex of $g$ either lies to the left of $e \cap I_1$, or is in $e \cap I_1$, or lies between $e \cap I_1$ and $e \cap I_2$.
Similarly, the right vertex of $g$ is either to the right of $f \cap  I_3$, or is in $f \cap I_3$ or lies between $f \cap I_3$ and $f \cap I_2$.
This gives us nine pairwise nonisomorphic ordered graphs formed by the edges $g$, $e$, and $f$.
Each of these ordered graphs is in Figure~\ref{fig_dichotomy}.

To finish the proof, note that every color in the coloring of $\mathcal{K}_{2^c}$ from the proof of Proposition~\ref{prop_star_lower} with $r_1=\dots=r_c=2=s_1=\dots=s_c$ induces an ordered subgraph of $2^c \cdot \mathcal{K}_{2^c,2^c}$.
In particular, there is no monochromatic copy of~$\mathcal{G}$.
Therefore we have $\Ro(\mathcal{G};c) > 2^c$.
\qed

\paragraph*{} We note that the coloring of $\mathcal{K}_{2^c}$ from the previous proof is an ordered variant of a particular ``optimal'' decomposition of the edges of $K_{2^c}$ into $c$ bipartite graphs.

\section{Open problems}\label{section_open_problems}

Theorem~\ref{thmIntChromDiag} implies that ordered Ramsey numbers of graphs that have bounded degree and bounded interval chromatic number are polynomial in the number of vertices.
However, we have no nontrivial lower bounds.

\begin{problem}
Is there an absolute constant $c>0$ such that for every fixed $\Delta$ there is a sequence $\{\mathcal{G}_n\}_{n\in\mathbb{N}}$ of ordered $\Delta$-regular graphs $\mathcal{G}_n$ with $n$ vertices and interval chromatic number $2$ such that $\Ro(\mathcal{G}_n) \ge n^{c\Delta}$?
\end{problem}

Similarly, it would be interesting to find some nontrivial lower bounds on ordered Ramsey numbers of ordered graphs of constant bandwidth.
Let $\mathcal{P}^{(p)}_n$ be the ordered graph on $n$ vertices $v_1,\dots,v_n$, in this order, such that $v_iv_j$ is an edge if and only if $0<|i-j|\leq p$. 
In particular, $\mathcal{P}^{(1)}_n=(P_n, \lhd_{mon})$.
Note that every ordered graph with $n$ vertices and with bandwidth at most $p$ is an ordered subgraph of $\mathcal{P}^{(p)}_n$.

\begin{problem}
For an integer $p \geq 2$, what is the growth rate of $\Ro(\mathcal{P}^{(p)}_n)$ with respect to $n$?
\end{problem}

We know that the ordered Ramsey numbers of alternating paths are linear with respect to the number of vertices.
Is it true that these orderings 
minimize ordered Ramsey numbers of ordered paths?

\begin{problem}
For some positive integer $n$, is there an ordering $\mathcal{P}_n$ of the path $P_n$ on $n$ vertices such that $\Ro(\mathcal{P}_n) < \Ro((P_n,\lhd_{alt}))$?
\end{problem}

Finally, we mention a problem with an application in the theory of geometric Ramsey numbers (see the definition in Section~\ref{subsection_motivation}).
A \emph{crossing} in an ordered graph $(G,\prec)$ is a pair of edges $v_iv_k$, $v_jv_l$ such that
$v_i \prec v_j \prec v_k \prec v_l$.
An ordered graph is \emph{noncrossing} if it contains no crossing.

Let $\Rnoncross(n)$ be the largest ordered Ramsey number of a noncrossing ordered graph on $n$ vertices.
Since noncrossing ordered graphs are outerplanar, they are $2$-degenerate, and thus, by a result of Conlon et al.~\cite{conFox14}, we have $\Rnoncross(n) \le n^{O(\log{n})}$.

\begin{problem}
What is the growth rate of $\Rnoncross(n)$? 
In particular, is it polynomial in $n$?
\label{que-crossing}
\end{problem}

It is an open problem whether there is a general polynomial upper bound for geometric Ramsey numbers of outerplanar graphs~\cite{cibGao13}.
The following theorem shows that Problem~\ref{que-crossing} is equivalent to the question of determining the asymptotics of the maximum convex geometric Ramsey number of an outerplanar graph on $n$ vertices.

\begin{theorem}
\label{theorem_relation_ord_convex_geom}
Let $\Rc(n)$ be the maximum convex geometric Ramsey number of an outerplanar graph on $n$ vertices.
For every $n \ge 2$, we have
\[
\Rc(n) \le \Rnoncross(n) \le \Rc(4n-4).
\]
\end{theorem}

\begin{proof}

Let $G$ be an outerplanar graph drawn in the plane so that its vertices are the vertices of a convex $n$-gon, and the edges are drawn as straight-line segments with no crossings.
Let $v_1 \prec v_2 \prec \dots \prec v_n$ be a clockwise ordering of the vertices of $G$ along the $n$-gon with $v_1$ chosen arbitrarily. 
In this way, we obtain a noncrossing ordered graph $\mathcal{G}$.
If we find a monochromatic copy of $\mathcal{G}$ in every $2$-coloring of $\mathcal{K}_N$
for some $N$, we can find a monochromatic noncrossing copy of the graph $G$ in every $2$-coloring
of the complete convex geometric graph on $N$ vertices.
This proves of the first inequality.

Now we prove the second inequality. 
The case $n = 2$ is trivial, so we assume that $n \ge 3$.
Since adding edges to an ordered graph never decreases its ordered Ramsey number, we know that $\Rnoncross(n)$
is attained by a noncrossing ordered graph $\mathcal{G}$ with vertices $v_1 \prec \dots \prec v_n$ that contains
the Hamiltonian cycle $v_1, v_2, \dots, v_n, v_1$.
We form an outerplanar graph $H$ as follows. 
We take four unordered copies $G^{(1)}, \dots, G^{(4)}$ of $\mathcal{G}$.
For every $i\in [4]$, let $v_1^{(i)}, \dots, v_n^{(i)}$ be the set of vertices of $G^{(i)}$.
We identify $v_n^{(1)}$ with $v_n^{(2)}$, $v_1^{(2)}$ with $v_1^{(3)}$, $v_n^{(3)}$ with $v_n^{(4)}$, and $v_1^{(4)}$ with $v_1^{(1)}$. See Figure~\ref{fig_geom_ram}.
The resulting graph $H$ is Hamiltonian and thus there is only one planar straight-line drawing of $H$ on a given set of $4n-4$ points 
in convex position, up to rotation and mirroring.

\begin{figure}
\centering
 \includegraphics{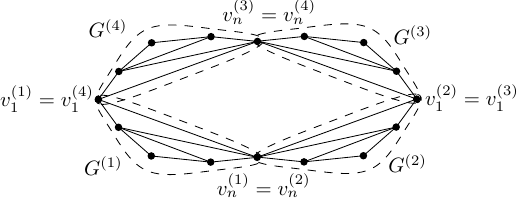} 
\caption{Construction of the graph $H$ in the proof of Theorem~\ref{theorem_relation_ord_convex_geom}.}
\label{fig_geom_ram}
\end{figure}

Let $K$ be a complete geometric graph whose vertices $u_1, u_2, \dots, u_N$ form, in this order, the vertices of a convex polygon.
In every noncrossing copy of $H$ in $K$, at least three of the graphs $G^{(i)}$, where $i \in [4]$, satisfy the property that the images of the vertices $v^{(i)}_1, \dots, v^{(i)}_n$ form a monotone sequence in the ordering $u_1 \prec u_2 \prec \dots \prec u_N$.
Consequently, in at least one $G^{(i)}$, the vertices form an increasing sequence.
If $N \ge \Rc(4n-4)$, every $2$-coloring of the complete convex geometric graph on $N$ vertices contains a monochromatic noncrossing copy of $H$. Therefore, every $2$-coloring of $\mathcal{K}_N$ contains a monochromatic copy of the ordered graph~$\mathcal{G}$.
\end{proof}

By the first inequality in Theorem~\ref{theorem_relation_ord_convex_geom}, the upper bound $\Rnoncross(n) \le n^{O(\log{n})}$ by Conlon et al.~\cite{conFox14} gives a quasipolynomial upper bound on $\Rc(n)$, improving the previous exponential bound (see, e.g.,~\cite{cibGao13}).

\section*{Acknowledgments}
The authors would like to thank to Ji\v{r}\'i Matou\v{s}ek for many helpful comments.
Part of the research was conducted during the DIMACS REU 2013 program.


\end{document}